\documentclass[12pt]{article}

\usepackage{amscd}
\usepackage{amsmath}
\usepackage{amssymb}
\usepackage{amstext}
\usepackage{amsthm}
\usepackage{bbold}
\usepackage{bm}
\usepackage{booktabs}
\usepackage{color}
\usepackage{colonequals}
\usepackage{comment}

\usepackage{easybmat}
\usepackage{etex}
\usepackage{enumitem}
\usepackage{framed}
\usepackage{authblk}
\usepackage[dvips,letterpaper,margin=1in]{geometry}
\usepackage{graphicx}
\usepackage{listings}
\usepackage{longtable}
\usepackage{mathtools}
\usepackage{multirow}
\usepackage{rotating}
\usepackage{setspace}
\usepackage{siunitx}
\usepackage{tabu}
\usepackage{verbatim}

\definecolor{cblack}{rgb}{0,0,0}
\definecolor{cblue}{rgb}{0.121569,0.466667,0.705882}    % 31,  119, 180
\definecolor{corange}{rgb}{1.000000,0.498039,0.054902}  % 256, 127, 14
\definecolor{cgreen}{rgb}{0.172549,0.627451,0.172549}   % 44,  160, 44
\definecolor{cred}{rgb}{0.839216,0.152941,0.156863}     % 214, 39,  40
\definecolor{cpurple}{rgb}{0.580392,0.403922,0.741176}  % 149, 103, 190
\definecolor{cbrown}{rgb}{0.549020,0.337255,0.294118}   % 141, 86,  75
\definecolor{cpink}{rgb}{0.890196,0.466667,0.760784}
\definecolor{cgray}{rgb}{0.498039,0.498039,0.498039}
\definecolor{cgreen2}{rgb}{0.7372549019607844, 0.7411764705882353, 0.13333333333333333}

\usepackage{hyperref}
\hypersetup{
  linkcolor  = cblue,
  citecolor  = cgreen,
  urlcolor   = corange,
  colorlinks = true,
}

\newtheorem{theorem}{Theorem}[section]
\newtheorem{remark}[theorem]{Remark}

\newtheorem{lemma}[theorem]{Lemma}

\newtheorem{definition}[theorem]{Definition}

\newtheorem{proposition}[theorem]{Proposition}

\newtheorem{conjecture}[theorem]{Conjecture}
\theoremstyle{plain} % just in case the style had changed
\newcommand{\thistheoremname}{}
\newtheorem*{genericthm}{\thistheoremname}

\newcommand{\Holder}{H\"{o}lder}

\newcommand{\Mobius}{M\"{o}bius}

\newcommand{\what}{\widehat}

\renewcommand{\emptyset}{\varnothing}

\makeatletter
\def\moverlay{\mathpalette\mov@rlay}
\def\mov@rlay#1#2{\leavevmode\vtop{%
   \baselineskip\z@skip \lineskiplimit-\maxdimen
   \ialign{\hfil$\m@th#1##$\hfil\cr#2\crcr}}}
\newcommand{\charfusion}[3][\mathord]{
    #1{\ifx#1\mathop\vphantom{#2}\fi
        \mathpalette\mov@rlay{#2\cr#3}
      }
    \ifx#1\mathop\expandafter\displaylimits\fi}
\makeatother

\newcommand{\CC}{\mathbb{C}}

\newcommand{\EE}{\mathbb{E}}

\newcommand{\NN}{\mathbb{N}}
\newcommand{\PP}{\mathbb{P}}

\newcommand{\RR}{\mathbb{R}}

\DeclareSymbolFont{bbold}{U}{bbold}{m}{n}
\DeclareSymbolFontAlphabet{\mathbbold}{bbold}
\newcommand{\One}{\mathbbold{1}}

\newcommand{\bv}{\bm v}

\newcommand{\bA}{\bm A}
\newcommand{\bB}{\bm B}
\newcommand{\bC}{\bm C}
\newcommand{\bD}{\bm D}

\newcommand{\bG}{\bm G}
\newcommand{\bH}{\bm H}

\newcommand{\bP}{\bm P}
\newcommand{\bQ}{\bm Q}

\newcommand{\bU}{\bm U}
\newcommand{\bV}{\bm V}

\newcommand{\bX}{\bm X}
\newcommand{\bY}{\bm Y}

\newcommand{\sA}{\mathcal{A}}

\newcommand{\sN}{\mathcal{N}}

\newcommand{\sU}{\mathcal{U}}

\newcommand{\Mod}[1]{\ (\mathrm{mod}\ #1)}

\DeclareSymbolFont{sfoperators}{OT1}{cmss}{m}{n}
\DeclareSymbolFontAlphabet{\mathsf}{sfoperators}
\makeatletter
\renewcommand{\operator@font}{\mathgroup\symsfoperators}
\makeatother

\DeclareMathOperator{\Part}{Part}

\DeclareMathOperator{\Tr}{Tr}

\DeclareMathOperator{\Unif}{Unif}

\DeclareMathOperator{\Var}{Var}

\newcommand{\Ex}{\mathop{\mathbb{E}}}  % puts symbols below
\newcommand{\Px}{\mathop{\mathbb{P}}}
\newcommand{\Ber}{\mathsf{Ber}}
\newcommand{\runs}{\mathsf{runs}}
\newcommand{\MANOVA}{\mathsf{MANOVA}}
\newcommand{\DE}{\mathsf{DE}}

\newcommand{\cyc}{\mathsf{cyc}}
\newcommand{\id}{\mathsf{id}}
\newcommand{\edge}{\mathsf{edge}}

\newcommand\numberthis{\addtocounter{equation}{1}\tag{\theequation}}

\title{Generic MANOVA limit theorems for products of projections}
\date{January 23, 2023}

\usepackage{authblk}
\author{Dmitriy Kunisky\thanks{Email: \texttt{dmitriy.kunisky@yale.edu}. Partially supported by ONR Award N00014-20-1-2335, a Simons Investigator Award to Daniel Spielman, and NSF grants DMS-1712730 and DMS-1719545.}}
\affil{Department of Computer Science, Yale University}

\begin{document}

\maketitle

\thispagestyle{empty}

\begin{abstract}
    We study the convergence of the empirical spectral distribution of $\bm A \bm B \bm A$ for $N \times N$ orthogonal projection matrices $\bm A$ and $\bm B$, where $\frac{1}{N}\Tr(\bm A)$ and $\frac{1}{N}\Tr(\bm B)$ converge as $N \to \infty$, to Wachter's MANOVA law.
    Using free probability, we show mild sufficient conditions for convergence in moments and in probability, and use this to prove a conjecture of Haikin, Zamir, and Gavish (2017) on random subsets of unit-norm tight frames.
    This result generalizes previous ones of Farrell (2011) and Magsino, Mixon, and Parshall (2021).
    We also derive an explicit recursion for the difference between the empirical moments $\frac{1}{N}\Tr((\bm A \bm B \bm A)^k)$ and the limiting MANOVA moments, and use this to prove a sufficient condition for convergence in probability of the largest eigenvalue of $\bm A \bm B \bm A$ to the right edge of the support of the limiting law in the special case where that law belongs to the Kesten-McKay family.
    As an application, we give a new proof of convergence in probability of the largest eigenvalue when $\bm B$ is unitarily invariant; equivalently, this determines the limiting operator norm of a rectangular submatrix of size $\frac{1}{2}N \times \alpha N$ of a Haar-distributed $N \times N$ unitary matrix for any $\alpha \in (0, 1)$.
    Unlike previous proofs, we use only moment calculations and non-asymptotic bounds on the unitary Weingarten function, which we believe should pave the way to analyzing the largest eigenvalue for products of random projections having other distributions.
\end{abstract}

\clearpage

\tableofcontents

\pagestyle{empty}

\clearpage

\setcounter{page}{1}
\pagestyle{plain}

\section{Introduction and main results}

The theory of free probability is a powerful framework for understanding the spectra of many constructions of random matrices.
However, the most general tools of free probability apply only to ``very random'' collections of matrices, that have, for example, been conjugated by Haar-distributed orthogonal or unitary transformations, ensuring that their frames of eigenvectors are correlated as little as possible.
It is natural to ask: are so many bits of randomness really necessary for various random matrix limit theorems to hold, or do these limits still hold for much more structured matrices?

In this paper, we will study one situation where a classical free probability result may be partly ``derandomized.''
Our work is inspired by a conjecture of~\cite{HZG-2017-MANOVA} on the spectra of random subsets of frames, which we will prove and extend in several directions.
Let us first present the free probability result that we will view their conjecture as derandomizing.\footnote{This point of view on the conjecture of \cite{HZG-2017-MANOVA} does not appear to have been explicitly stated in prior work, though similar earlier results \cite{TVCS-2007-GaussianErasureChannel, TCSV-2010-CapacityChannelsFadingWachter, FN-2013-TruncationsTensorHaarUnitary, AF-2014-AsymptoticallyLiberatingUnitaryMatrices} used the tools of free probability, and \cite{MMP-2019-RandomSubensembles}, who proved part of the conjecture of \cite{HZG-2017-MANOVA}, noticed that similar combinatorial objects arise in their calculations as those underlying free probability computations (in particular, summations over non-crossing partitions).}
The result concerns the following family of probability measures.
\begin{definition}[Wachter's MANOVA law \cite{Wachter-1980-EmpiricalMeasureDiscriminantRatios}]
    \label{def:manova}
    For two parameters $\alpha, \beta \in (0, 1)$, we say a random variable has the law $\MANOVA(\alpha, \beta) = \MANOVA(\beta, \alpha)$ if it has the following density with respect to Lebesgue measure:
    \begin{equation}
        d\mu(x) = \frac{\sqrt{(r_+ - x)(x - r_-)}}{2\pi x(1 - x)} \One_{[r_-, r_+]}(x)dx + (1 - \min\{\alpha, \beta\})\delta_0(x) + \max\{\alpha + \beta - 1, 0\}\delta_1(x),
    \end{equation}
    where $\delta_c$ denotes a Dirac mass at $c$, and
    \begin{align*}
      r_{\pm} = r_{\pm}(\alpha, \beta) &\colonequals \alpha + \beta - 2\alpha\beta \pm 2\sqrt{\alpha(1 - \alpha)\beta(1 - \beta)} \\
              &= \left(\sqrt{\alpha(1 - \beta)} \pm \sqrt{\beta(1 - \alpha)}\right)^2.
    \end{align*}
    Let us also write
    \begin{equation}
        \edge(\alpha, \beta) \colonequals \left\{ \begin{array}{ll} 1 & \text{if } \alpha + \beta > 1, \\
                                                    r_+(\alpha, \beta) & \text{if } \alpha + \beta \leq 1 \end{array} \right\},
                                        \end{equation}
                                        which is the right edge of the support of $\MANOVA(\alpha, \beta)$.
\end{definition}

\begin{remark}[MANOVA parametrizations]
    We take our parametrization of the MANOVA distributions from the free probability literature cited below; see also the discussion in~\cite{CK-2014-LiberationProjections} for an overview.
    Some works in signal processing, such as~\cite{Farrell-2011-LimitingEmpirical}, have also used this parametrization.
    This parametrization seems especially natural to us in highlighting the symmetric roles of $\alpha$ and $\beta$.
    However, some results in this area use different parametrizations, whose relations to this one we explain in detail in Appendix~\ref{app:manova-param}.
\end{remark}

The following is a straightforward computation when we use the tools of free probability; see the references below or our Section~\ref{sec:free-prob} for details.
For a Hermitian matrix $\bX \in \CC^{N \times N}$ with eigenvalues $\lambda_1, \dots, \lambda_N \in \RR$, we call $\frac{1}{N}\sum_{i = 1}^N \delta_{\lambda_i}$ its \emph{empirical spectral distribution (e.s.d.)}.

\begin{definition}[Convergence of random measures]
    \label{def:measure-conv}
    For $\mu_i$ random probability measures on $\RR$ (and its Borel $\sigma$-algebra) for $i \in \NN$ and $\mu$ a deterministic probability measure on $\RR$, we say that the $\mu_i$ \emph{converge in moments} to $\mu$ if, for all $k \geq 1$, $\lim_{i \to \infty} \EE \int x^k d\mu_i(x) = \int x^k d\mu(x)$ with the right-hand side finite.
    We say that $\mu_i$ \emph{converge in probability (respectively, almost surely)} to $\mu$ if, for every continuous and compactly supported function $f: \RR \to \RR$, $\int f(x) d\mu_i(x) \to \int f(x) d\mu(x)$ in probability (respectively, almost surely).
\end{definition}
\begin{theorem}[Example 3.6.7 of \cite{VDN-1992-FreeRandomVariables}, Exercise 18.27 of \cite{NS-2006-LecturesCombinatoricsFreeProbability}]
    \label{thm:invariant-model}
    Let $N = N(\iota) \in \NN$ be an increasing sequence.
    Let $\bA, \bB \in \CC^{N \times N}$ be random orthogonal projection matrices\footnote{The reader will recall that being an orthogonal projection matrix is equivalent to being a Hermitian projection matrix.} such that $\frac{1}{N}\Tr(\bA) \to \alpha$ and $\frac{1}{N}\Tr(\bB) \to \beta$ as $\iota \to \infty$ for some $\alpha, \beta \in (0, 1)$.
    Let $\bU \in \sU(N)$ be a Haar-distributed unitary matrix.
    Then, the e.s.d.\ of $\bU\bA \bU^{*} \bB \bU \bA \bU^{*}$ converges in moments to $\MANOVA(\alpha, \beta)$.
\end{theorem}
\noindent
The underlying phenomenon in the language of free probability is that the (sequences of) matrices $\bU\bA\bU^{*}$ and $\bB$ are \emph{asymptotically free}, and the MANOVA law is the limiting law of $(\bU\bA\bU^{*})^{1/2}\bB(\bU\bA\bU^{*})^{1/2} = \bU\bA\bU^{*}\bB\bU\bA\bU^{*}$ by virtue of being the \emph{free multiplicative convolution} of the limiting e.s.d.'s of $\bA$ and $\bB$, which are the Bernoulli distributions $\Ber(\alpha)$ and $\Ber(\beta)$, respectively.
In standard free probability notation,
\begin{equation}
    \label{eq:manova-free-conv}
    \MANOVA(\alpha, \beta) = \Ber(\alpha) \boxtimes \Ber(\beta).
\end{equation}
This much is assured by the technology of free probability; with a little more effort, one may also show that the same convergence occurs in probability or, taking a suitable subsequence of $N$, almost surely.
See Remark~\ref{rem:almost-sure} and Section~\ref{sec:moment-background} for further discussion.

\subsection{Weak convergence of empirical spectral distribution}

The following, our first main result, gives a much weaker condition on a sequence of pairs of random projections under which the conclusion of Theorem~\ref{thm:invariant-model} still holds.
\begin{theorem}
    \label{thm:convergence-moments}
    Let $N = N(\iota)$ be an increasing sequence.
    Let $\bA, \bB \in \CC^{N \times N}$ be random orthogonal projection matrices.
    Suppose that, for some $\alpha, \beta \in (0, 1)$,
    \begin{align}
      \lim_{\iota \to \infty} \frac{1}{N} \EE \Tr(\bA - \alpha \bm I_N) = \lim_{\iota \to \infty} \left(\frac{1}{N} \EE \Tr(\bA) - \alpha\right) &= 0, \\
      \lim_{\iota \to \infty} \frac{1}{N} \EE \Tr(\bB - \beta \bm I_N) = \lim_{\iota \to \infty} \left(\frac{1}{N} \EE \Tr(\bB) - \beta\right) &= 0,
      \intertext{and, for all $k \geq 1$,}
      \lim_{\iota \to \infty} \frac{1}{N} \EE \Tr\big((\bA - \alpha \bm I_N)(\bB - \beta \bm I_N)\big)^k &= 0. \label{eq:asymp-free-relation}
    \end{align}
    Then, the e.s.d.\ of $\bA\bB\bA$ converges in moments to $\MANOVA(\alpha, \beta)$.
\end{theorem}

\begin{remark}[Asymptotic freeness]
    In fact, our proof will imply the much more consequential statement that, under the assumptions of Theorem~\ref{thm:convergence-moments}, $\bA$ and $\bB$ are, just as in the setting of Theorem~\ref{thm:invariant-model}, asymptotically free.
    The crux of the matter is just that, while for general matrices asymptotic freeness describes a large number of limits of expectations of polynomials, for the special case of projections, thanks to their idempotence, asymptotic freeness reduces to just the one-parameter family of limits \eqref{eq:asymp-free-relation}.
    When these hold, free probability also allows us to compute the limiting spectral moments of polynomials in $\bA$ and $\bB$ such as the sum $\bA + \bB$, the Hermitian commutator $(\bA\bB - \bB\bA)^2$, and so forth.
\end{remark}

We also provide a stronger version of our result that guarantees convergence in probability.
Conveniently, it simply suffices to replace all convergences in expectation above with convergences in $L^2$.
\begin{theorem}
    \label{thm:convergence-prob}
    Let $N = N(\iota)$ be an increasing sequence.
    Let $\bA, \bB \in \CC^{N \times N}$ be random orthogonal projection matrices.
    Suppose that, for some $\alpha, \beta \in (0, 1)$,
    \begin{align}
      \lim_{\iota \to \infty} \EE \left(\frac{1}{N}\Tr(\bA - \alpha \bm I_N)\right)^2 = \lim_{\iota \to \infty} \EE \left(\frac{1}{N}\Tr(\bA) - \alpha\right)^2 &= 0, \\
      \lim_{\iota \to \infty} \EE \left(\frac{1}{N}\Tr(\bB - \beta \bm I_N)\right)^2 = \lim_{\iota \to \infty} \EE \left(\frac{1}{N}\Tr(\bB) - \beta\right)^2 &= 0,
      \intertext{and, for all $k \geq 1$,}
      \lim_{\iota \to \infty} \EE \left(\frac{1}{N}\Tr\big((\bA - \alpha \bm I_N)(\bB - \beta \bm I_N)\big)^k\right)^2 &= 0. \label{eq:asymp-free-relation}
    \end{align}
    Then, the e.s.d.\ of $\bA\bB\bA$ converges in probability to $\MANOVA(\alpha, \beta)$.
\end{theorem}

Compared to Theorem~\ref{thm:invariant-model}, the conjecture of \cite{HZG-2017-MANOVA} proposes reducing the amount of randomness in $\bU\bA\bU^{*}$, a projection to a uniformly random subspace of the same dimension as the image of $\bA$.
Instead, they consider $\bA$ a random \emph{coordinate} projection, i.e., a diagonal projection matrix.
In this case, their conjecture says that, for a large class of deterministic projection matrices $\bB$, the spectrum of such $\bA\bB\bA$ still has a MANOVA limit.

The $\bB$ that the conjecture pertains to originate in frame theory.
\begin{definition}[Frames]
    A \emph{frame} is a collection of vectors $\bv_1, \dots, \bv_N \in \CC^M$.
    We identify a frame with the matrix $\bV \in \CC^{M \times N}$ whose columns are the $\bv_i$.
    A frame has \emph{unit norm} if $\|\bv_i\| = 1$ for all $i \in [N]$, is \emph{tight} if $\sum_{i = 1}^N \bv_i\bv_i^{*} = \beta^{-1} \bm I_M$ for some $\beta > 0$, and is \emph{equiangular} if $|\langle \bv_i, \bv_j \rangle| = c$ for some $c \geq 0$ for all $i, j \in [N]$ distinct.
    We abbreviate \emph{UNTF} for unit-norm tight frames, and \emph{ETF} for equiangular unit-norm tight frames (in particular, every ETF is a UNTF).
\end{definition}
\noindent
For $\bV$ a UNTF, the Gram matrix $\bV^{*}\bV$ is a rescaled projection matrix: it has the same non-zero spectrum as $\bV \bV^{*} = \sum_{i = 1}^N \bv_i\bv_i^{*} = \frac{N}{M} \bm I_M$, where we identify $\beta = \frac{M}{N}$ by taking traces on either side.
Thus, $\bB = \frac{M}{N}\bV^*\bV$ is a projection matrix for any $\bV$ a UNTF, which has $\frac{1}{N}\Tr(\bB) = \beta$.

It was observed by \cite{HZG-2017-MANOVA} that a wide range of $\bB$ built in this way from UNTFs appear to satisfy the conclusion of Theorem~\ref{thm:invariant-model}.
They did not propose definitive criteria for precisely which $\bB$ should satisfy this conclusion, but they conjectured that this should be the case when $\bB$ is built from $\bV$ an ETF, and also perhaps any $\bV$ a UNTF that is ``close enough'' to being equiangular.

Our Theorem~\ref{thm:convergence-prob} almost immediately proves the conjecture of \cite{HZG-2017-MANOVA} for the broad class of UNTFs satisfying only a weak quantitative \emph{incoherence} condition, including all ETFs (see the discussion after the result below).
We may also slightly weaken the unit norm assumption, allowing small fluctuations in the norms of the frame vectors, and allow the frame vectors to be random.
\begin{theorem}
    \label{thm:untf}
    Let $M = M(\iota)$ and $N = N(\iota)$ be increasing sequences with $\frac{M}{N} \to \beta \in (0, 1)$, and $\bV \in \CC^{M \times N}$ be random having $\bV \bV^* = \bm I_M$ (i.e., being a random tight frame).
    Write $\bv_1, \dots, \bv_N \in \CC^M$ for the columns of $\bV$.
    Suppose that there is a constant $C > 0$ such that, for any $k \geq 1$,
    \begin{equation}
        \label{eq:thm:untf:assumption}
        \limsup_{\iota \to \infty} N^{k/2 - C} \max_{i, j \in [N]} \EE[|\langle \bv_i, \bv_j \rangle - \One\{i = j\} \beta|^k] = 0.
    \end{equation}
    Let $\bA \in \RR^{N \times N}$ be a random diagonal matrix whose diagonal entries are drawn i.i.d.\ as $A_{ii} \sim \Ber(\alpha)$ for some $\alpha \in (0, 1)$.
    Then, the e.s.d.\ of $\bA \bV^*\bV \bA$ converges in probability to $\MANOVA(\alpha, \beta)$.
\end{theorem}
\noindent
Proving Theorem~\ref{thm:untf} is a routine matter of verifying the main assumption of Theorem~\ref{thm:convergence-prob} with a combinatorial calculation.

To help interpret Theorem~\ref{thm:untf}, let us discuss some special cases when this result does and does not apply.
\begin{enumerate}
\item The result must \emph{not} apply to $\bV$ an i.i.d.\ matrix of Gaussian random variables distributed as $\sN(0, \frac{1}{M})$, since the limit of the e.s.d.\ of such a random matrix is a Mar\v{c}enko-Pastur law rather than a MANOVA law (see the discussion in \cite{HZG-2017-MANOVA}).
    While the main assumption \eqref{eq:thm:untf:assumption} holds in this case, the assumption of \emph{exact} tightness $\bV\bV^* = \bm I_M$ is violated.
    Thus, it is not enough for this assumption to hold only approximately or in expectation.
\item When the $\bV$ are a deterministic sequence, it suffices to have the incoherence condition
\begin{equation}
    \limsup_{\iota \to \infty} \sqrt{N}\, f(N) \cdot \max_{i, j \in [N]} \left| \langle \bv_i, \bv_j \rangle - \One\{i = j\} \beta \right| < \infty
\end{equation}
for some subpolynomial $f(N) = \exp(o(\log N))$.
In particular, the conjecture of \cite{HZG-2017-MANOVA, HGMZ-2021-AsymptoticFrameTheory} in the case of $\bV$ a sequence of deterministic ETFs follows immediately, as does the prior result of \cite{MMP-2019-RandomSubensembles} for ETFs in the special case $\beta = \frac{1}{2}$.
This is because, for an ETF, we may compute that all mutual angles for $i, j \in [N]$ distinct are
\begin{equation}
    |\langle \bv_i, \bv_j \rangle| = \sqrt{\frac{N - M}{M(N - 1)}} = \Theta(N^{-1/2}),
\end{equation}
by the saturation of the \emph{Welch bound} (see, e.g., Chapter 12 of \cite{Waldron-2018-FiniteTightFrames}).
\item As for $\bV$ random, the results of \cite{Farrell-2011-LimitingEmpirical} on $\bV$ formed by taking a random subset of columns of either the discrete Fourier transform matrix or a Haar-distributed unitary matrix follow readily, as does the result implicitly given by \cite{FN-2013-TruncationsTensorHaarUnitary} for $\bV$ a random subset of columns of a small tensor power of a Haar-distributed unitary matrix.
\end{enumerate}

\begin{remark}[Almost sure convergence]
    \label{rem:almost-sure}
    We do not work with almost sure convergence of random measures (see Definition~\ref{def:measure-conv}) for the sake of simplicity and since this does not seem relevant to applications.
    However, a version of Theorem~\ref{thm:untf} for almost sure convergence follows from the same proof we give and a variant of Lemma~\ref{lem:conv-prob-var}, so long as $N$ is restricted to a lacunary subsequence.
    A similar argument is given by \cite{MMP-2019-RandomSubensembles}.
\end{remark}

More generally, we believe Theorems~\ref{thm:convergence-moments} and \ref{thm:convergence-prob} should pave the way towards verifying that the conclusion of Theorem~\ref{thm:invariant-model} holds even in some completely deterministic settings.
For example, even the randomness in $\bA$ in Theorem~\ref{thm:untf} should not always be necessary for the MANOVA limit to hold.
In a follow-up work, we will explore this phenomenon and its consequences for the spectra of subgraphs of the Paley graph of number theory.

\subsection{Convergence of maximum eigenvalue in Kesten-McKay case}

We also give some results on the convergence of the largest eigenvalue of products of projections, $\lambda_{\max}(\bA\bB\bA)$.
As is usual in random matrix theory, such results are more technical to prove than weak convergence results, because when using trace calculations they require the analysis of $\Tr(\bA\bB\bA)^k$ where $k$ diverges sufficiently quickly with $N$ rather than being constant.
Classical results of this kind are those of~\cite{FK-1981-EigenvaluesRandomMatrices} treating Wigner random matrices and~\cite{Geman-1980-LimitNormRandomMatrix} treating Wishart random matrices.

While we believe that analogs of the results below should apply to arbitrary $\alpha, \beta \in (0, 1)$, we focus on the case where one of $\alpha$ or $\beta$ equals $\frac{1}{2}$.
In this case, the limiting $\MANOVA(\alpha, \beta)$ measure is (after suitable translation and rescaling) a member of the special subfamily of \emph{Kesten-McKay} measures, which simplifies the requisite combinatorial calculations.

\begin{remark}[Role of Kesten-McKay measure]
    The reader may be surprised to find the Kesten-McKay measure appearing in limit theorems for \emph{dense} random matrices, since it also appears as a limiting law for the spectra of adjacency matrices of \emph{sparse} random graphs (and trees of constant degree).
    One account of why this happens, using the language of free probability, is that the Kesten-McKay measures appear both as certain free \emph{additive} and free \emph{multiplicative} convolutions of Bernoulli measures.
    The former accounts for the relevance to sparse random matrices, and the latter for the relevance to dense random matrices.
    An exposition of both interpretations may be found in \cite{LM-2020-FreeProbabilityKestenMcKay}.
\end{remark}

\begin{theorem}
    \label{thm:edge-km}
    Let $N = N(\iota)$ be an increasing sequence.
    Let $\bA, \bB \in \CC^{N \times N}$ be random orthogonal projection matrices.
    Suppose that $\frac{1}{N} \Tr(\bA) \to \alpha \in (0, 1)$ and $\frac{1}{N}\Tr(\bB) \to \beta \in (0, 1)$ as $\iota \to \infty$.
    Suppose that one of $\alpha$ or $\beta$ equals $\frac{1}{2}$, that the e.s.d.\ of $\bA\bB\bA$ converges in probability to $\MANOVA(\alpha, \beta)$, and that there is also a sequence $k = k(\iota)$ such that $k \gg \log N$ and for which
    \begin{equation}
        \label{eq:max-eval-cond}
        \max_{1 \leq a \leq k} \left| \EE \Tr\left(\frac{\bA - \alpha \bm I_N}{\sqrt{\alpha(1 - \alpha)}} \frac{\bB - \beta \bm I_N}{\sqrt{\beta(1 - \beta)}}\right)^a \right| \leq \exp(o(k)).
    \end{equation}
    Then, $\lambda_{\max}(\bA\bB\bA) = \|\bA\bB\|^2$ converges in probability to $\edge(\alpha, \beta)$.
\end{theorem}

The condition~\eqref{eq:max-eval-cond} is a natural ``tameness'' condition on $\bA$ and $\bB$: when $\EE\bA = \alpha \bm I_N$, then the normalization $\what{\bA} \colonequals (\bA - \alpha \bm I_N) / \sqrt{\alpha(1 - \alpha)}$ is ``orthogonal in expectation,'' in the sense that $\EE\what{\bA} = \bm 0$ and $\EE \what{\bA}^2 = \bm I_N$, and likewise for $\what{\bB} \colonequals (\bB - \beta \bm I_N) / \sqrt{\beta(1 - \beta)}$ when $\EE \bB = \beta \bm I_N$.
Indeed, if $\beta = \frac{1}{2}$, then $\what{\bB} = 2\bB - \bm I_N$ which is an \emph{exactly} orthogonal symmetric matrix, being a reflection through the hyperplane to which $\bB$ projects.
The condition~\eqref{eq:max-eval-cond} states that, even when $\alpha \neq \frac{1}{2}$ and $\what{\bA}$ is not actually orthogonal, traces of powers of $\what{\bA}\what{\bB}$ satisfy a qualitatively similar growth to those of an orthogonal matrix (which would be uniformly bounded by $N = \exp(o(k))$ when $k \gg \log N$).

As we will discuss later, it seems likely that Theorem~\ref{thm:edge-km} can be strengthened to apply without the ``Kesten-McKay condition'' that one of $\alpha$ or $\beta$ equals $\frac{1}{2}$.
The main obstacle to proving such a generalization is adequately understanding certain combinatorial polynomials appearing in the recursive description of moments of the e.s.d.\ of $\bA\bB\bA$; see Remark~\ref{rem:max-eval-manova} for further details.

\begin{conjecture}
    \label{conj:edge-manova}
    Theorem~\ref{thm:edge-km} holds for arbitrary $\alpha, \beta \in (0, 1)$.
\end{conjecture}

As an application, we use Theorem~\ref{thm:edge-km} to give a new proof of a limit theorem for the largest eigenvalue when $\bB$ is unitarily invariant, giving an edge limit theorem counterpart to the weak limit theorem claim in the $\beta = \frac{1}{2}$ case of Theorem~\ref{thm:invariant-model}.
\begin{theorem}
    \label{thm:edge-km-unitary}
    Let $\bU \in \sU(N)$ be a Haar-distributed unitary matrix.
    Let $\bA, \bD \in \{0, 1\}^{N \times N}$ be random diagonal matrices with $A_{ii} \sim \Ber(\alpha)$ and $D_{ii} \sim \Ber(\frac{1}{2})$ independently.
    Then, $\lambda_{\max}(\bA \bU \bD \bU^{*} \bA) = \| \bA \bU \bD \|^2$ converges in probability to $\edge(\alpha, \frac{1}{2})$ as $N \to \infty$.
\end{theorem}
\noindent
Note that $\bA\bU\bD$ is a random rectangular submatrix of $\bU$, where each row is included with probability $\alpha$ and each column with probability $\frac{1}{2}$.
Note also that this result fits the earlier framework when we set $\bB \colonequals \bU \bD \bU^{*} = (\bU\bD)(\bU\bD)^*$, an orthogonal projection matrix to the span of the columns of $\bU$ ``picked out'' by $\bD$.

Much more detailed versions of Theorem~\ref{thm:edge-km-unitary}, showing convergence of the entire distribution of a rescaling of $\lambda_{\max}(\bA \bU \bD \bU^{*} \bA)$, have been obtained before by \cite{Collins-2005-ProductRandomProjections, Johnstone-2008-JacobiMatrixLargestEigenvalue}.
However, we believe that our proof of Theorem~\ref{thm:edge-km-unitary} using Theorem~\ref{thm:edge-km} provides new ideas for analyzing the largest eigenvalues of submatrices of many interesting structured matrices beyond this preliminary result, such the conference matrices studied by \cite{MMP-2019-RandomSubensembles} (which have $\beta = \frac{1}{2}$ and so can be analyzed---in principle---using Theorem~\ref{thm:edge-km}) or the Fourier matrices studied by \cite{Farrell-2011-LimitingEmpirical}.

\subsection{Related work}
\label{sec:related}

\paragraph{Jacobi ensembles}
A closely related random matrix ensemble to a product of random projections, of more significance to MANOVA techniques of statistics, takes the form
\begin{equation}
    \label{eq:jacobi-mx}
    (\bG\bG^* + \bH\bH^*)^{-1/2} \bH\bH^* (\bG\bG^* + \bH\bH^*)^{-1/2}
\end{equation}
for $\bG$ and $\bH$ matrices with i.i.d.\ standard complex Gaussian entries of dimensions $N \times L$ and $N \times M$ respectively, with $L + M \geq N$.
This law (or that of the spectrum of such a matrix) is known in random matrix theory as the \emph{(unitary) Jacobi ensemble}, and its statistics may be understood in great detail thanks to the possibility of many closed-form Gaussian computations and connections with Jacobi polynomials \cite{CC-2004-AsymptoticFreenessGaussianWishart, Collins-2005-ProductRandomProjections, Demni-2008-FreeJacobiProcess, Johnstone-2008-JacobiMatrixLargestEigenvalue, DK-2008-ExtremeEigenvaluesJacobi, DZ-2009-JacobiProcessLargeDeviations, DHH-2012-FreeJacobiProcessSpectral, EF-2013-MANOVALocalDensity, CK-2014-LiberationProjections}.
See also the introduction of \cite{Jiang-2013-LimitTheoremsJacobiEnsembles} for a more thorough survey.

As observed by Wachter \cite{Wachter-1980-EmpiricalMeasureDiscriminantRatios}, and later again by \cite{Collins-2005-ProductRandomProjections,AEK-2006-PrincipalAngleRandomSubspaces,ES-2008-BetaJacobiMatrixModel}, the law of the non-zero spectrum of a product $\bA\bB\bA$ of unitarily invariant projections\footnote{We say this as a shorthand for \emph{either} $\bA$ or $\bB$ being unitarily invariant: if, e.g., $\bB$ is invariant having the same law as $\bU\bB\bU^*$ for $\bU$ a Haar-distributed unitary, then the spectrum of $\bA\bB\bA$ has the same law as that of $\bQ\bA\bU\bB\bU^*\bA\bQ^*$ for $\bQ$ another Haar-distributed unitary, which in turn has the same law as $\bQ\bA\bQ^*\bU\bB\bU^*\bQ\bA\bQ^*$ since $(\bQ, \bU)$ has the same law as $(\bQ, \bQ^*\bU)$.} is actually \emph{exactly equal} (for matrices of fixed size) to the law of the spectrum a Jacobi ensemble matrix.
Thus, in the special unitarily invariant case, the many powerful results on Jacobi ensembles also apply immediately to products of projections.

\paragraph{``Less random'' settings}
The prior work on the spectrum of a product of projections $\bA\bB\bA$ where $\bA$ and $\bB$ are ``less random'' or fully deterministic is surveyed by \cite{HGMZ-2021-AsymptoticFrameTheory}.
One important result, due to \cite{Farrell-2011-LimitingEmpirical}, established the MANOVA limit for random submatrices of the discrete Fourier transform matrix (i.e., $\bA$ a random coordinate projection and $\bB = \bU\bD\bU^*$ for $\bD$ a random coordinate projection and $\bU$ the Fourier matrix).
Note that $\bU\bD$, with zero columns removed, is a UNTF in this case.
Similar limits, though not quite belonging to our setting (as they have $\bU$ a Toeplitz matrix that is \emph{not} necessarily unitary, whereby $\bB$ is not necessarily a projection), are studied by \cite{TVCS-2007-GaussianErasureChannel,TCSV-2010-CapacityChannelsFadingWachter}.
Another similar result of \cite{FN-2013-TruncationsTensorHaarUnitary} established the MANOVA limit for random submatrices of a tensor power of a Haar-distributed unitary matrix.
\cite{RK-2016-MANOVAColumnSubsampled} established a general result similar to Theorem~\ref{thm:untf}, but under a condition of ``norm-concentration'' (their Definition 1) on $\bB$ that seems impractical to verify for deterministic $\bB$.

Inspired by some of this work, \cite{HZG-2017-MANOVA} conjectured a form of our Theorem~\ref{thm:untf}, substantiated by numerical experiments.
This was proved for the special case of ETFs when the MANOVA limit is a Kesten-McKay law by \cite{MMP-2019-RandomSubensembles}.
As we show in Appendix~\ref{app:asymp-lib}, a variation of that result also follows from earlier work in free probability of \cite{AF-2014-AsymptoticallyLiberatingUnitaryMatrices}.
Later,~\cite{HZ-2021-MomentsSubsetsETF} derived a combinatorial conjecture that would imply the result for general ETFs (corresponding to convergence of empirical moments to MANOVA moments), and verified the first few of these symbolically.

\paragraph{The largest eigenvalue}
To the best of our knowledge, no limit theorems are known for the largest eigenvalues of products of projections in settings where the weak limit of the e.s.d.\ is a MANOVA law, aside from the special case of unitarily invariant projections.
Other cases are raised as open questions by \cite{Farrell-2011-LimitingEmpirical,MMP-2019-RandomSubensembles, HGMZ-2021-AsymptoticFrameTheory}.
For the unitarily invariant case, \cite{Collins-2005-ProductRandomProjections, Johnstone-2008-JacobiMatrixLargestEigenvalue} established Tracy-Widom limit theorems for the fluctuations of the largest eigenvalue of such matrices, results that subsume and greatly extend our Theorem~\ref{thm:edge-km-unitary}.
Explicit non-asymptotic formulas for the distribution of these fluctuations can also be derived in this case \cite{AEK-2006-PrincipalAngleRandomSubspaces,DK-2008-ExtremeEigenvaluesJacobi}.

There seems to be some confusion about the state of this literature: \cite{HZG-2017-MANOVA} suggest that there is no known limit theorem on the largest eigenvalue of a product of unitarily invariant projections (the ``Haar frame'' case, in their terminology).
But in fact, using the distributional equivalence between products of projections and Jacobi ensembles mentioned above, such a limit theorem follows immediately from the results of \cite{Collins-2005-ProductRandomProjections,Johnstone-2008-JacobiMatrixLargestEigenvalue}.
We too only learned of this argument after arriving at our proof of Theorem~\ref{thm:edge-km-unitary}, and hope for this discussion to clarify the matter.

On the other hand, we are not aware of a proof of such a limit theorem using traces of powers in the style of the classical results of \cite{FK-1981-EigenvaluesRandomMatrices,Geman-1980-LimitNormRandomMatrix}.
This proof technique seems far more robust, not requiring any special ``integrability'' of the underlying model.

\subsection{Open problems}

We highlight a few problems that our work leaves open as future research directions.

\begin{enumerate}
\item Prove Conjecture~\ref{conj:edge-manova}. Together with Lemma~\ref{lem:invariant-max-eval-bd}, this would give a new combinatorial proof that the operator norm of a random $\alpha N \times \beta N$ submatrix of a Haar-distributed $N \times N$ unitary matrix converges in probability to $\sqrt{\edge(\alpha, \beta)}$ for any $\alpha, \beta \in (0, 1)$.
\item Establish the limit in probability of the maximum eigenvalue for random subframes of various UNTFs, such as Paley ETFs, random Fourier frames, or random columns of tensor powers of a Haar-distributed unitary matrix (for which weak convergence of the e.s.d.\ was studied by \cite{MMP-2019-RandomSubensembles,Farrell-2011-LimitingEmpirical, FN-2013-TruncationsTensorHaarUnitary}, respectively).
\item Investigate when limit theorems on the spectrum or the maximum eigenvalue hold in entirely deterministic versions of the settings studied here. We will investigate one instance of this question as pertains to subframes of Paley ETFs (or, equivalently, subgraphs of the Paley graph of number theory) in a follow-up work.
\item Establish limit theorems for the fluctuations of the maximum eigenvalue (as done by \cite{Collins-2005-ProductRandomProjections, Johnstone-2008-JacobiMatrixLargestEigenvalue} for Jacobi ensembles) for any ensemble discussed here that is not equivalent to a Jacobi ensemble.
\item Develop sufficient conditions (perhaps similar to Theorem~\ref{thm:edge-km}) for the minimum non-zero eigenvalue of any of the ensembles discussed here that is not equivalent to a Jacobi ensemble to converge in probability to $r_{-}(\alpha, \beta)$, the left edge of the absolutely continuous part of $\MANOVA(\alpha, \beta)$.
\end{enumerate}

\subsection{Notation}

We often consider an increasing sequence $N = N(\iota) \in \NN$ and sometimes $M = M(\iota) \in \NN$ of dimensions for our matrices or frames, where $\iota \in \NN$ is reserved as an index variable and usually left implicit.
In this setting, all asymptotic notations ($O(\cdot), o(\cdot), \Theta(\cdot), \ll, \gg, \lesssim, \gtrsim$) refer to the limit $\iota \to \infty$.
Subscripts of these notations specify the variables on which implicit constants depend.

$\bX^*$ denotes the conjugate transpose of a complex-valued matrix $\bX$.
$\sU(N)$ denotes the group of $N \times N$ unitary matrices.
We use the convention that matrix powers bind before the trace, so that, for example, $\Tr (\bX\bY)^a = \Tr((\bX\bY)^a)$.
$\lambda_{\max}(\cdot)$ denotes the largest eigenvalue of a matrix, and $\| \cdot \|$ the operator norm or largest singular value.

$\Ber(\alpha)$ denotes the law of a Bernoulli variable equal to 1 with probability $\alpha$ and to 0 with probability $1 - \alpha$.
$\EE[\cdot]$ denotes the expectation and $\Var[\cdot]$ the variance of a random variable.

We use the standard notation $[n] = \{1, \dots, n\}$.
For a set $S$, we write $\binom{S}{k}$ for the set of subsets of $S$ of size $k$.

$S_n$ denotes the symmetric group of permutations of $[n]$.
For $\sigma \in S_n$, we write $\cyc(\sigma)$ for the set of cycles of $\sigma$ (viewed as sets, not permutations), and $\cyc_a(\sigma)$, $\cyc_{\geq a}(\sigma)$, and $\cyc_{\leq a}(\sigma)$ for the sets of cycles of size exactly $a$, at least $a$, and at most $a$, respectively.

$\Part([n])$ denotes the set of partitions of the set $[n]$ (not the number $n$).
We view a partition $\rho \in \Part([n])$ as a set of sets, so that $|\rho|$ denotes the size of the partition.
For two partitions $\pi, \rho \in \Part([n])$, we write $\rho \geq \pi$ is $\rho$ is coarser than $\pi$ (i.e., if every part of $\pi$ is contained in a part of $\rho$).

\subsection{Organization}

The remainder of the paper is organized as follows.

In Section~\ref{sec:convergence-moments}, we prove and discuss Theorem~\ref{thm:convergence-moments} on conditions for convergence in moments to a MANOVA limit.
This uses tools from free probability, which we introduce in Section~\ref{sec:free-prob}.

In Section~\ref{sec:convergence-prob}, we prove Theorem~\ref{thm:convergence-prob} on conditions for convergence in probability to a MANOVA limit and, as an application, prove Theorem~\ref{thm:untf} on the spectrum of random subframes of incoherent UNTFs.

In Section~\ref{sec:moment-recursions}, preparing to study the maximum eigenvalue of $\bA\bB\bA$, we derive explicit recursions for the moments of the e.s.d.\ of such a matrix, for the moments of MANOVA distributions, and for the differences between the two.

In Section~\ref{sec:max-eval-km}, using these recursions, we prove Theorem~\ref{thm:edge-km} and, as an application, prove Theorem~\ref{thm:edge-km-unitary} on the largest eigenvalue of a product of unitarily invariant projections, or, equivalently, the operator norm of random rectangular submatrices of Haar-distributed unitary matrices.

Two of our appendices also contain elaborations of the main results that the reader may find especially interesting.

In Appendix~\ref{app:gf}, we present generating function calculations that both are used in our results in Section~\ref{sec:max-eval-km} and give an alternative proof of part of our results in Section~\ref{sec:convergence-moments} without relying explicitly on free probability reasoning. This kind of proof is more in line with the prior work of \cite{MMP-2019-RandomSubensembles, HZ-2021-MomentsSubsetsETF}, verifying ``by hand'' that the MANOVA moments satisfy a particular combinatorial recursion.

In Appendix~\ref{app:asymp-lib}, we show how the ideas of \cite{AF-2014-AsymptoticallyLiberatingUnitaryMatrices} on ``asymptotically liberating'' matrices in free probability may be used to produce a proof of a variant of Theorem~\ref{thm:untf} in the special case where one of $\alpha$ or $\beta$ equals $\frac{1}{2}$, so that $\MANOVA(\alpha, \beta)$ is a Kesten-McKay law.

\section{Convergence in moments to MANOVA}
\label{sec:convergence-moments}

\subsection{Background on free probability}
\label{sec:free-prob}

We very briefly recall some of the central ideas and definitions of free probability.
The reader may consult references such as \cite{VDN-1992-FreeRandomVariables,ENV-1998-SaintFlourNotes,TV-2004-RandomMatrixWirelessCommunications,NS-2006-LecturesCombinatoricsFreeProbability,Novak-2014-ThreeLectures,MS-2017-FreeProbabilityRandomMatrices} for much more information.
\paragraph{Generalities}
A \emph{noncommutative probability space} is a unital algebra $\sA$ over $\CC$ endowed with a linear functional $\phi: \sA \to \CC$ called a \emph{state} which satisfies, for $\id_{\sA}$ the identity element of $\sA$, the normalization $\phi(\id_{\sA}) = 1$.
If $\phi(\mathfrak{a}\mathfrak{b}) = \phi(\mathfrak{b}\mathfrak{a})$ for all $\mathfrak{a}, \mathfrak{b} \in \sA$, then $\phi$ and the space $(\sA, \phi)$ are called \emph{tracial}.
We will assume all noncommutative probability spaces under discussion are tracial.

\paragraph{Convergence}
We now present several notions concerning pairs of elements of a noncommutative probability space; these notions may all be generalized to greater numbers of elements and sometimes more broadly to subalgebras of $\sA$, but we will not need these more general formulations.
The \emph{law} of $\mathfrak{a} \in \sA$ is the state $\psi$ on the algebra $\CC[a]$ given by $\psi(f(a)) \colonequals \phi(f(\mathfrak{a}))$.
Similarly, the \emph{joint law} of $\mathfrak{a}, \mathfrak{b} \in \sA$ is the state $\psi$ on the algebra $\CC\langle a, b \rangle$ (the noncommutative polynomials in two indeterminates $a, b$) given by $\psi(f(a, b)) \colonequals \phi(f(\mathfrak{a}, \mathfrak{b}))$ for every noncommutative polynomial $f$.
The idea is that these laws remember the moments of $\mathfrak{a}$ or the joint moments of $\mathfrak{a}, \mathfrak{b}$ but discard all other information about $\sA$ (such as the dimensions of matrices).
That allows us to define a useful abstract notion of convergence.
Suppose that $(\sA^{(N)}, \phi^{(N)})$ for $N \in \NN$ are noncommutative probability spaces, and $\mathfrak{a}^{(N)}, \mathfrak{b}^{(N)} \in \sA^{(N)}$.
We say $\mathfrak{a}^{(N)}$ is \emph{convergent in law} to a state $\psi$ on $\CC[\mathfrak{a}]$ if the laws of these elements converge weakly to $\psi$, i.e., if $\lim_{N \to \infty} \phi^{(N)}(f(\mathfrak{a}^{(N)})) = \psi(f(\mathfrak{a}))$ for all $f$.
Similarly, we say the sequence of pairs $\mathfrak{a}^{(N)}, \mathfrak{b}^{(N)}$ is convergent in law to a state $\psi$ on $\CC\langle \mathfrak{a}, \mathfrak{b} \rangle$ if the joint laws converge weakly to $\psi$, i.e., if $\lim_{N \to \infty} \phi^{(N)}(f(\mathfrak{a}^{(N)}, \mathfrak{b}^{(N)})) = \psi(f(\mathfrak{a}, \mathfrak{b}))$ for all $f$.

\paragraph{Freeness}
Two elements $\mathfrak{a}, \mathfrak{b} \in \sA$ are called \emph{free} if, for any polynomials $p_1, q_1, \dots, p_k, q_k \in \CC[z]$ such that $\phi(p_i(\mathfrak{a})) = \phi(q_j(\mathfrak{b})) = 0$, we also have $\phi(p_1(\mathfrak{a})q_1(\mathfrak{b}) \cdots p_k(\mathfrak{a})q_k(\mathfrak{b})) = 0$.
Note that if $\mathfrak{a}, \mathfrak{b}$ are free, then their joint law is determined by their individual laws.
In particular, we may define $\CC\langle \mathfrak{a}, \mathfrak{b} \rangle$ as a noncommutative probability space by declaring that $\mathfrak{a}$ and $\mathfrak{b}$ are free, and specifying a law for $\mathfrak{a}$ and $\mathfrak{b}$ individually.
In the setting above with sequences $\mathfrak{a}^{(N)}, \mathfrak{b}^{(N)}$, we say that these pairs are \emph{asymptotically free} if they are convergent in law to a state $\psi$ on $\CC\langle \mathfrak{a}, \mathfrak{b} \rangle$ under which $\mathfrak{a}$ and $\mathfrak{b}$ are free.

\paragraph{Random matrix setting}
For our purposes, we will always consider an increasing sequence $N = N(\iota) \in \NN$, and will work over noncommutative probability spaces $\sA^{(\iota)}$ of random variables (suitable $\sigma$-algebras will be clear from context) taking values in $\CC^{N(\iota) \times N(\iota)}$, endowed with the ``normalized trace'' state $\phi^{(\iota)}(\bA) \colonequals \frac{1}{N(\iota)} \EE\Tr(\bA)$.

\paragraph{Free convolution}
For the purposes of random matrix theory, the benefit of the above definitions is that, when sequences of matrices are asymptotically free, then we may reason about their joint moments in terms of the limiting free elements $\mathfrak{a}, \mathfrak{b}$.
The following is one important manifestation of this.
Suppose that $\mathfrak{a}^{(N)}$ is convergent in law to a state $\psi_a$ on $\CC[a]$, and that there exists a probability measure $\mu$ on $\RR$ such that $\psi_a(f(a)) = \EE_{a \sim \mu}[f(a)]$ for each polynomial $f$.
Suppose the same holds for $\mathfrak{b}^{(N)}$ for a probability measure $\nu$ and a state $\psi_b$ on $\CC[b]$.
Suppose also that each $\mathfrak{a}^{(N)}$ has a square root, $\mathfrak{a}^{(N)} = (\sqrt{\mathfrak{a}^{(N)}})^2$.
If $\mathfrak{a}^{(N)}, \mathfrak{b}^{(N)}$ are asymptotically free, then $\sqrt{\mathfrak{a}^{(N)}}\, \mathfrak{b}^{(N)}\sqrt{\mathfrak{a}^{(N)}}$ is also convergent in law to a state $\psi$ on $\CC[z]$, and there exists a measure, denoted $\mu \boxtimes \nu$ and called the \emph{free multiplicative convolution} of $\mu$ and $\nu$, such that $\psi(f(z)) = \EE_{z \sim \mu \, \boxtimes \, \nu}[f(z)]$.
This measure may be determined by computations using the \emph{$S$-transform}; see Section 4.5.3 of \cite{MS-2017-FreeProbabilityRandomMatrices}.

\paragraph{Unitary invariance}
Finally, let us recall one of the foundational results of free probability, due to Voiculescu.
We say that the law of a random matrix $\bA \in \CC^{N \times N}$ is \emph{unitarily invariant} if $\bA$ has the same law as $\bU\bA\bU^*$ for any $\bU \in \sU(N)$.
The result states that unitary invariance is sufficient to give asymptotic freeness.

\begin{theorem}[\cite{Voiculescu-1991-RandomMatricesFreeProducts}]
    \label{thm:invariant-freeness}
    Let $N = N(\iota)$ be an increasing sequence, and $\bA^{(N)}, \bB^{(N)} \in \CC^{N \times N}$ be random matrices.
    Suppose that the sequences $\bA^{(N)}$ and $\bB^{(N)}$ are individually convergent in law (i.e., that the e.s.d.'s of the $\bA^{(N)}$ converge in moments, and likewise those of the $\bB^{(N)}$), and that every matrix in one of the sequences $\bA^{(N)}$ or $\bB^{(N)}$ is unitarily invariant.
    Then, the sequence of pairs $\bA^{(N)}, \bB^{(N)}$ is asymptotically free.
    They are convergent in law to a state on $\CC \langle \mathfrak{a}, \mathfrak{b} \rangle$ under which $\mathfrak{a}$ has the limiting law of $\bA^{(N)}$, $\mathfrak{b}$ has the limiting law of $\bB^{(N)}$, and $\mathfrak{a}$ and $\mathfrak{b}$ are free.
\end{theorem}

\subsection{Sufficient condition: Proof of Theorem~\ref{thm:convergence-moments}}

We will give a total of three treatments of Theorem~\ref{thm:convergence-moments}.
First, we will discuss how free probability pertains to orthogonal projection matrices and show how the Theorem follows as a consequence of those considerations.
We will also give a more direct though less conceptual derivation using only the consequence of this free probability machinery of Theorem~\ref{thm:invariant-model}.
Finally, in Appendix~\ref{app:gf}, we will also give an explicit combinatorial proof (when combined with the results of Section~\ref{sec:moment-recursions}) that is closer to the spirit of the approach to Theorem~\ref{thm:untf} pursued by \cite{MMP-2019-RandomSubensembles,HZ-2021-MomentsSubsetsETF}.

Let us discuss asymptotic freeness for orthogonal projection matrices in particular.
The following are the natural limiting objects of projection matrices.
\begin{definition}
    We say that $\mathfrak{a} \in \sA$ a noncommutative probability space with state $\phi$ is a \emph{Bernoulli element with parameter $\alpha$} if $\mathfrak{a}^2 = \mathfrak{a}$ and $\phi(\mathfrak{a}) = \alpha$.
    Note that these conditions determine the law of $\mathfrak{a}$.
\end{definition}
\noindent
We observe that, for projection matrices, the condition of asymptotic freeness is simpler than in the general case.
\begin{proposition}
    \label{prop:asymp-free-proj}
    Let $N = N(\iota)$ be an increasing sequence, and $\bA^{(N)}, \bB^{(N)} \in \CC^{N \times N}$ be random orthogonal projection matrices with $\frac{1}{N}\Tr(\bA^{(N)}) \to \alpha$ and $\frac{1}{N}\Tr(\bB^{(N)}) \to \beta$ for some $\alpha, \beta \in (0, 1)$ as $\iota \to \infty$.
    Suppose that, for each $k \geq 1$,
    \begin{equation}
        \label{eq:asymp-free-proj}
    \lim_{\iota \to \infty} \frac{1}{N} \EE \Tr\big((\bA^{(N)} - \alpha \bm I_N)(\bB^{(N)} - \beta \bm I_N)\big)^k = 0.
    \end{equation}
    Then, the sequence of pairs $\bA^{(N)}$, $\bB^{(N)}$ is asymptotically free.
    They are convergent in law to a state on $\CC \langle \mathfrak{a}, \mathfrak{b} \rangle$ under which $\mathfrak{a}, \mathfrak{b}$ are Bernoulli elements with parameters $\alpha, \beta$ respectively, and $\mathfrak{a}$ and $\mathfrak{b}$ are free.
\end{proposition}
\begin{proof}
    Because $\bA^{(N)}$ and $\bB^{(N)}$ are idempotent, their joint law is determined by traces of the form $\frac{1}{N}\Tr(\bA^{(N)}\bB^{(N)})^k$.
    The condition \eqref{eq:asymp-free-proj}, upon expanding the multiplications and applying idempotence, determines the limits $\lim_{N \to \infty}\frac{1}{N}\Tr(\bA^{(N)}\bB^{(N)})^k$ recursively as functions of $\alpha, \beta$.

    On the other hand, in $\sA \colonequals \CC\langle \mathfrak{a}, \mathfrak{b} \rangle$, by freeness we have, for each $k \geq 1$,
    \begin{equation}
        \phi\Big(\big((\mathfrak{a} - \alpha \, \id_{\sA})(\mathfrak{b} - \beta \, \id_{\sA})\big)^k\Big) = 0.
    \end{equation}
    In particular, since $\mathfrak{a}$ and $\mathfrak{b}$ are also idempotent, the same recursion applies to the $\phi((\mathfrak{a}\mathfrak{b})^k)$ as to the $\lim_{N \to \infty}\frac{1}{N}\Tr(\bA^{(N)}\bB^{(N)})^k$.
    Therefore, we must in fact have
    \begin{equation}
        \lim_{N \to \infty}\frac{1}{N}\Tr((\bA^{(N)}\bB^{(N)})^k) = \phi((\mathfrak{a}\mathfrak{b})^k).
    \end{equation}
    By idempotence and linearity of states this gives that $\bA^{(N)}, \bB^{(N)}$ converge in law to $\mathfrak{a}, \mathfrak{b}$.
\end{proof}

Finally, let us state formally the free convolution characterization of the MANOVA law, which we mentioned above.
We note that this is not a difficult result, requiring only some computations with suitable transforms of the measures involved.
\begin{proposition}[Example 3.6.7 of \cite{VDN-1992-FreeRandomVariables}, Exercise 18.27 of \cite{NS-2006-LecturesCombinatoricsFreeProbability}]
    \label{prop:manova-free-conv}
    $\MANOVA(\alpha, \beta) = \Ber(\alpha) \boxtimes \Ber(\beta)$.
\end{proposition}
\noindent
Note that with the tools we have introduced, we may prove Theorem~\ref{thm:invariant-model}: it is merely a combination of Theorem~\ref{thm:invariant-freeness}, which shows asymptotic freeness of the matrices involved, our discussion of free multiplicative convolution, which gives that the measure that the e.s.d.\ of $(\bU\bA\bU^*)^{1/2} \bB (\bU\bA\bU^*)^{1/2} = \bU\bA\bU^* \bB \bU\bA\bU^*$ converges to in moments is $\Ber(\alpha) \boxtimes \Ber(\beta)$, and Proposition~\ref{prop:manova-free-conv}, which identifies this measure as $\MANOVA(\alpha, \beta)$.

We are now prepared to give the most principled proof of our claim.
\begin{proof}[Proof 1 of Theorem~\ref{thm:convergence-moments}]
    By Proposition~\ref{prop:asymp-free-proj}, under the assumptions of the Theorem, the sequence of $\bA, \bB$ is asymptotically free, converging to free Bernoulli elements $\mathfrak{a}, \mathfrak{b}$ with parameters $\alpha, \beta$, respectively.
    In particular, $\bA$ converges in law to $\mathfrak{a}$, whose moments coincide with those of the measure $\Ber(\alpha)$, and likewise for $\bB, \mathfrak{b}$, and the measure $\Ber(\beta)$.
    From our discussion of free multiplicative convolution, we therefore have that $\bA^{1/2}\bB\bA^{1/2} = \bA \bB\bA$ converges in law to $\sqrt{\mathfrak{a}}\, \mathfrak{b}\sqrt{\mathfrak{a}} = \mathfrak{a}\mathfrak{b}\mathfrak{a}$, whose moments coincide with those of the measure $\Ber(\alpha) \boxtimes \Ber(\beta)$, which equals $\MANOVA(\alpha, \beta)$ by Proposition~\ref{prop:manova-free-conv}.
    The result then follows from observing that convergence in law of a sequence of random matrices is identical to convergence in moments of the e.s.d.\ of those matrices.
\end{proof}

It is also possible to rewrite this proof in a more ``down-to-earth'' fashion, avoiding the mention of limiting free Bernoulli elements $\mathfrak{a}, \mathfrak{b}$, which may be simpler to parse for the reader unfamiliar with free probability.
We emphasize, however, that all we are doing below is working more directly with the definition and basic properties of free multiplicative convolution.

\begin{proof}[Proof 2 of Theorem~\ref{thm:convergence-moments}]
    We first note that the assumption on joint moments of $\bA$ and $\bB$,
    \begin{equation}
        \lim_{\iota \to \infty} \frac{1}{N} \EE \Tr \big((\bA - \alpha \bm I_N)(\bB - \beta \bm I_N)\big)^k = 0, \label{eq:convergence-moments-condition-2}
    \end{equation}
    determines the limiting moments of the e.s.d.\ of $\bA\bB\bA$.
Indeed, expanding the polynomial in the trace of the asymptotic freeness relation and using the idempotence of $\bA$ as needed, we find
\begin{equation}
    \Tr \big((\bA - \alpha \bm I_N)(\bB - \beta \bm I_N)\big)^k = v_{k, 0} \Tr(\bm I_N) + v_{k, 1}\Tr(\bA) + v_{k, 2} \Tr(\bB) + \sum_{\ell = 1}^k w_{k, \ell} \Tr(\bA\bB\bA)^{\ell}, \label{eq:error-moment-general}
\end{equation}
and the coefficients of this relation are functions only of $\alpha$ and $\beta$: $v_{k, j} = v_{k, j}(\alpha, \beta)$ and $w_{k, \ell} = w_{k, \ell}(\alpha, \beta)$.
These in particular also satisfy that $w_{k, k} = 1$.
Thus, by induction we find that each limit $\lim_{\iota \to \infty} \frac{1}{N} \EE\Tr (\bA\bB\bA)^k$ exists for each $k \geq 1$, and is a function of $\alpha$ and $\beta$ determined by the above relations.

On the other hand, for $\bA$ and $\bB$ unitarily invariant as in the setting of Theorem~\ref{thm:invariant-model} (with the same $\alpha$ and $\beta$), these relations hold as well, as a consequence of asymptotic freeness (which follows from Theorem~\ref{thm:invariant-freeness}).
We know that in that case the e.s.d.\ of $\bA\bB\bA$ converges in moments to $\MANOVA(\alpha, \beta)$ by Theorem~\ref{thm:invariant-model}, and therefore $\bA\bB\bA$ for any $\bA, \bB$ satisfying \eqref{eq:convergence-moments-condition-2} must also converge in moments to $\MANOVA(\alpha, \beta)$.
\end{proof}

\section{Convergence in probability to MANOVA}
\label{sec:convergence-prob}

\subsection{Sufficient condition: Proof of Theorem~\ref{thm:convergence-prob}}
\label{sec:moment-background}

Our proof will use the following well-known condition allowing convergence in moments to be extended to convergence in probability.
\begin{lemma}[Exercise 2.4.6 of \cite{Tao-2012-RandomMatrixTheory}, Section 2.1.2 of \cite{AGZ-2010-RandomMatrices}]
    \label{lem:conv-prob-var}
    Let $N = N(\iota)$ be an increasing sequence.
    Let $\bA \in \CC^{N \times N}$ be a sequence of random Hermitian matrices with $\|\bA\|$ uniformly bounded with probability 1 and whose e.s.d.'s converge in moments to some bounded measure $\mu$ on $\RR$.
    If moreover for each $k \geq 1$
    \begin{equation}
        \lim_{\iota \to \infty} \Var\left[\frac{1}{N} \Tr(\bA^k)\right] = 0,
    \end{equation}
    then the same convergence holds in probability.
\end{lemma}

\begin{proof}[Proof of Theorem~\ref{thm:convergence-prob}]
    First, note that the assumptions subsume those of Theorem~\ref{thm:convergence-moments}, so we have by that result that the e.s.d.'s of $\bA\bB\bA$ converge to $\MANOVA(\alpha, \beta)$ in moments.
    We will show later in Lemma~\ref{lem:error-term-recursion} that, for every $k \geq 1$, there is a constant $m_k = m_k(\alpha, \beta)$ such that $\Tr(\bA\bB\bA)^k - m_k N$ is a linear combination of $\Tr(\bA) - \alpha N$, $\Tr(\bB - \beta N)$, and $\Tr((\bA - \alpha \bm I_N)(\bB - \beta \bm I_N))^{\ell}$ for $1 \leq \ell \leq k$ with coefficients depending only on $\alpha, \beta,$ and $k$.
    Indeed, here $m_k$ will just be the $k$th moment of $\MANOVA(\alpha, \beta)$.
    (The result of Lemma~\ref{lem:error-term-recursion} is just a more specific and inverted version of the relation \eqref{eq:error-moment-general} mentioned above.)
    We may then bound coarsely
    \begin{align*}
      \Var\left[\frac{1}{N} \Tr((\bA\bB\bA)^k)\right]
      &= \Var\left[\frac{1}{N} \Tr((\bA\bB\bA)^k) - m_k\right] \\
      &\lesssim_{k} \Var\left[\frac{1}{N}(\Tr(\bA) - \alpha N)\right] + \Var\left[\frac{1}{N}(\Tr(\bB) - \beta N)\right] \\
      &\hspace{1cm} + \sum_{\ell = 1}^k\Var\left[\frac{1}{N}\Tr((\bA - \alpha \bm I_N)(\bB - \beta \bm I_N))^{\ell}\right] \\
      &\leq \EE\left(\frac{1}{N}(\Tr(\bA) - \alpha N)\right)^2 + \EE\left(\frac{1}{N}(\Tr(\bB) - \beta N)\right)^2 \\
      &\hspace{1cm} + \sum_{\ell = 1}^k\EE\left(\frac{1}{N}\Tr((\bA - \alpha \bm I_N)(\bB - \beta \bm I_N))^{\ell}\right)^2, \numberthis
    \end{align*}
    and the result follows using Lemma~\ref{lem:conv-prob-var}, since our assumptions imply that each term tends to zero as $\iota \to \infty$.
\end{proof}

We note that \cite{MMP-2019-RandomSubensembles} proved a result for the setting of Theorem~\ref{thm:untf} for ETFs with $\beta = \frac{1}{2}$ allowing convergence in moments to be extended to convergence in probability using Talagrand's concentration inequality for Lipschitz functions of independent bounded random variables (the independent variables in Theorem~\ref{thm:untf} being the diagonal entries of $\bA$).
Also, \cite{HGMZ-2021-AsymptoticFrameTheory} extended this to apply to general UNTFs in the same setting.
However, instead of relying on these tools below we will use our Theorem~\ref{thm:convergence-prob}, to illustrate that it is not difficult to use and to keep our treatment self-contained (in particular, avoiding relying on the powerful machinery of Talagrand's inequality).

On the other hand, those results worked with almost sure convergence, while ours only guarantee convergence in probability.
But, using a suitable variation of Lemma~\ref{lem:conv-prob-var} (see Exercise 2.4.6 of \cite{Tao-2012-RandomMatrixTheory}) and allowing for a subsequence of $\iota$ to be taken, it is straightforward to extend Theorem~\ref{thm:untf} to almost sure convergence.

\subsection{Application to UNTFs: Proof of Theorem~\ref{thm:untf}}
\label{sec:pf:thm:untf}

\begin{proof}[Proof of Theorem~\ref{thm:untf}]
    We establish that the condition of Theorem~\ref{thm:convergence-prob} holds.
    Let us write $\widetilde{\bA} \colonequals \bA - \alpha \bm I_N$ and $\widetilde{\bB} \colonequals \bB - \beta \bm I_N$.
    Recall that $\bA$ and thus also $\widetilde{\bA}$ are both diagonal.
    Recall also our assumption on the moments of $\bB$, which with the above notation may be written as, for some $C > 0$ and all $k \geq 1$,
    \begin{equation}
        \max_{i, j \in [N]} \EE|\widetilde{B}_{ij}|^k = o_k(N^{-k/2 + C}).
    \end{equation}
    By applying Jensen's inequality $\EE|\widetilde{B}_{ij}|^k \leq (\EE|\widetilde{B}_{ij}|^{tk})^{1/t} = o_k(N^{-k/2 + C/t})$ using a sufficiently large $t \geq 1$ depending on $C$, we may assume without loss of generality that $C = 1$.

    We expand
    \begin{align*}
      &\hspace{-1cm}\EE \left(\Tr\big((\bA - \alpha \bm I_N)(\bB - \beta \bm I_N)\big)^k\right)^2 \\
      &= \EE \left( \Tr(\widetilde{\bA}\widetilde{\bB})^k\right)^2 \\
      &= \EE \sum_{\substack{i_1, \dots, i_k \in [N]^k \\ j_1, \dots, j_k \in [N]^k}} \widetilde{A}_{i_1i_1} \cdots \widetilde{A}_{i_ki_k} \widetilde{A}_{j_1j_1} \cdots \widetilde{A}_{j_kj_k} \widetilde{B}_{i_1i_2} \cdots \widetilde{B}_{i_ki_1} \widetilde{B}_{j_1j_2} \cdots \widetilde{B}_{j_kj_1} \\
      &\leq \sum_{\substack{i_1, \dots, i_k \in [N]^k \\ j_1, \dots, j_k \in [N]^k}} \big|\EE\widetilde{A}_{i_1i_1} \cdots \widetilde{A}_{i_ki_k} \widetilde{A}_{j_1j_1} \cdots \widetilde{A}_{j_kj_k} \big| \cdot \EE |\widetilde{B}_{i_1i_2} \cdots \widetilde{B}_{i_ki_1} \widetilde{B}_{j_1j_2} \cdots \widetilde{B}_{j_kj_1}|
      \intertext{By \Holder's inequality on the $\bB$ expectation,}
      &\leq \sum_{\substack{i_1, \dots, i_k \in [N]^k \\ j_1, \dots, j_k \in [N]^k}} \big|\EE\widetilde{A}_{i_1i_1} \cdots \widetilde{A}_{i_ki_k} \widetilde{A}_{j_1j_1} \cdots \widetilde{A}_{j_kj_k} \big| \cdot \\
      &\hspace{2.45cm} (\EE |\widetilde{B}_{i_1i_2}|^{2k})^{1/2k} \cdots (\EE |\widetilde{B}_{i_ki_1}|^{2k})^{1/2k} (\EE |\widetilde{B}_{j_1j_2}|^{2k})^{1/2k} \cdots (\EE |\widetilde{B}_{j_kj_1}|^{2k})^{1/2k}
      \intertext{Here, any term where each index appearing among the $i_1, \dots, i_k, j_1, \dots, j_k$ does not appear at least twice will equal zero, since the $\widetilde{A}_{ii}$ are independent and centered. Moreover, the magnitude of every non-zero term is $o_k(N^{-2k/2 + 1}) = o_k(N^{-k + 1})$ by our assumption on the moments of entries of $\bB$ and since $|\widetilde{A}_{ii}| = |A_{ii} - \alpha| \leq 1$ with probability 1. Therefore,}
      &\ll_k N^{-k + 1} \#\{(i_1, \dots, i_{2k}) \in [N]^{2k}: \text{ each index occurs at least twice}\} \\
      &\leq N^{-k + 1} \sum_{\substack{\pi \in \Part([2k]) \\ |S| \geq 2 \text{ for all } S \in \pi}} N^{|\pi|} \\
      &\leq N|\Part([2k])| \\
      &\lesssim_k N. \numberthis
    \end{align*}
    Thus, Theorem~\ref{thm:convergence-prob} applies and gives the result.
\end{proof}

\section{Explicit moment recursions}
\label{sec:moment-recursions}

We now develop explicit combinatorial expressions for the recursions we have used abstractly (e.g., in the second proof of Theorem~\ref{thm:convergence-moments}) for the traces $\Tr(\bA\bB\bA)^k$.
These will be crucial in our analysis of the right edge of the spectrum of $\bA\bB\bA$, for which we will need to understand the magnitude of the coefficients in these recursions.

\subsection{Idempotent necklace expansion}

The main tool we apply is a deterministic combinatorial recursion for traces of powers of \emph{any} product of projection matrices.
We first define some preliminary combinatorial objects.
\begin{definition}
    \label{def:set-partition}
    Suppose $S \subseteq [k]$ is non-empty with $|S| = j$ and $S = \{s_1 < \cdots < s_j\}$.
    We write $\bm p = (p_1, \dots, p_j) = (p_1(S), \dots, p_j(S))$ for an ordered partition of $k$ induced by $S$, given by $p_1 = s_2 - s_1, \dots, p_{j - 1} = s_j - s_{j - 1}, p_j = k + s_1 - s_j$.
    Note that each $p_i > 0$.
\end{definition}
\noindent
It is best to think of this partitioning as happening on the $k$-cycle: we may think for example of $S$ as a subset of edges of the cycle that are deleted, and the partition $\bm p$ as the numbers of vertices in each resulting path.
(Note that since $S$ is non-empty, the result is indeed always a union of paths and never the entire cycle.)

We will see that the following, superficially rather convoluted, definition will come up naturally in expanding the traces of powers.
In Appendix~\ref{app:gf} we show that these polynomials are surprisingly well-behaved, in particular having a convenient rational ordinary generating function.
\begin{definition}
    \label{def:q-poly}
    For $0 < j \leq k$ and $0 \leq a \leq j$, define the polynomials
    \begin{equation}
        q_{k, j, a}(x) \colonequals \sum_{S \in \binom{[k]}{j}} \sum_{A \in \binom{[j]}{a}} \prod_{i = 1}^j \left(\sum_{b = 0}^{p_i(S) - 1}(-1)^{p_i(S) - 1 - b} x^{b} + (-1)^{p_i(S)}\One\{i \notin A\}\right).
    \end{equation}
    For $j = a = 0$, we define $q_{k, 0, 0}(x) \colonequals 0$.
\end{definition}

\begin{lemma}
    \label{lem:necklace-expansion}
    Let $\bA, \bB \in \CC^{N \times N}$ be orthogonal projection matrices, and $\alpha, \beta \in \RR$.
    Then, we have
    \begin{equation}
        \Tr(\bA\bB\bA)^k = \beta^k\sum_{\ell = 0}^k \binom{k}{\ell}\widetilde{\mu}_{\ell},
    \end{equation}
    where the $\widetilde{\mu}_{k}$ satisfy the recursion
    \begin{align*}
      \widetilde{\mu}_0 &= \Tr(\bA), \numberthis \\
      \widetilde{\mu}_k &= \alpha^k \left(\sum_{b = 1}^{k - 1}(1 - \beta^{-1})^b\right) N - \sum_{a = 0}^{k - 1} \widetilde{\mu}_{a} \sum_{j = 1}^{k - 1} (-\alpha)^{k - j} q_{k, j, a}(\beta^{-1} - 1) \\
                        &\hspace{2cm} + (-1)^{k + 1} \frac{\alpha^k}{\beta}\left(\sum_{b = 0}^{k - 1}(1 - \beta^{-1})^b\right)\left(\frac{1}{N}\Tr(\bB) - \beta \right)N \\
      &\hspace{2cm} + \frac{1}{\beta^k}\Tr\big((\bA - \alpha \bm I_N)(\bB - \beta \bm I_N)\big)^k \text{ for } k \geq 1. \numberthis
    \end{align*}
\end{lemma}

\noindent
The details of the coefficients aside, there is a simple algebraic reason for a formula like this to hold (essentially the same as we alluded to in the second proof of Theorem~\ref{thm:convergence-moments}): consider rewriting $\Tr(\bA\bB\bA)^k = \Tr(\bA\bB)^k = \Tr((\alpha\bm I_N + \widetilde{\bA})(\beta \bm I_N + \widetilde{\bB}))^k$ for $\widetilde{\bA} \colonequals \bA - \alpha \bm I_N$ and $\widetilde{\bB} \colonequals \bB - \beta\bm I_N$, and then expanding the product of sums into a sum of products.
The result will be a sum of terms of the form $\alpha^s\beta^t\Tr(\widetilde{\bA}^{i_1} \widetilde{\bB}^{j_1} \cdots \widetilde{\bA}^{i_k}\widetilde{\bB}^{j_k})$ for some $i_a, j_b \in \{0, 1\}$, $s = k - \sum_{a = 1}^k i_a$, and $t = k - \sum_{b = 1}^k j_b$.
The quadratic equation $\bA^2 = \bA$ may be rewritten into another quadratic equation for $\widetilde{\bA}$, and similarly for $\bB$ and $\widetilde{\bB}$.
Thus, we may repeatedly rewrite each such term until we reach one of four kinds of traces: $\Tr(\bm I_N) = N$, $\Tr(\widetilde{\bA}) = \Tr(\bA) - \alpha N$, $\Tr(\widetilde{\bB}) = \Tr(\bB) - \beta N$, and $\Tr(\widetilde{\bA}\widetilde{\bB})^{\ell}$ for some $\ell \leq k$.
The Lemma moreover describes the specific coefficients with which each of these ``terminal'' kinds of traces appears.

\begin{proof}[Proof of Lemma~\ref{lem:necklace-expansion}]
    Define $\widetilde{\bA} \colonequals \bA - \alpha \bm I_N$ and $\widetilde{\bB} \colonequals \bB - \beta \bm I_N$.
    We begin by expressing the moments we want to calculate in terms of similar moments with $\bB$ replaced by $\widetilde{\bB}$.
    We have
    \begin{align*}
      \Tr(\bA\bB\bA)^k
      &= \Tr(\bA (\beta \bm I_N + \widetilde{\bB}) \bA)^k
        \intertext{and, expanding by the binomial theorem and using that $\bA$ is idempotent,}
      &= \sum_{\ell = 0}^k \binom{k}{\ell}\beta^{k - \ell} \Tr((\bA\widetilde{\bB})^{\ell}\bA) \\
      &= \beta^k \Tr(\bA) + \beta^k\sum_{\ell = 1}^k \binom{k}{\ell} \Tr(\bA\beta^{-1}\widetilde{\bB})^{\ell}. \numberthis
    \end{align*}
    Thus it suffices to show that
    \begin{align}
      \widetilde{\mu}_0 &\colonequals \Tr(\bA), \\
      \widetilde{\mu}_{k} &\colonequals \Tr(\bA\beta^{-1}\widetilde{\bB})^{k} = \Tr((\alpha \bm I_N + \widetilde{\bA})\beta^{-1}\widetilde{\bB})^{k} \text{ for } k \geq 1
    \end{align}
    satisfy the recursion in the statement.
    As indicated earlier, our strategy will be to expand the product remaining in $\widetilde{\mu}_k$ above.

    In doing so, we will end up with powers of $\beta^{-1}\widetilde{\bB}$.
    This matrix satisfies a quadratic equation: since
    \begin{equation}
        \beta \bm I_N + \widetilde{\bB} = \bB = \bB^2 = (\beta \bm I_N + \widetilde{\bB})^2,
    \end{equation}
    we may expand and reorganize into
    \begin{equation}
        (\beta^{-1}\widetilde{\bB})^2 = (\underbrace{\beta^{-1} - 2}_{= \, \eta - 1}) \beta^{-1}\widetilde{\bB} + (\underbrace{\beta^{-1} - 1}_{\equalscolon \, \eta}) \bm I_N.
    \end{equation}
    More generally, by repeatedly applying this identity, we will have $(\beta^{-1}\widetilde{\bB})^k = c_k\widetilde{\bB} + d_k \bm I_N$, and these coefficients will satisfy the recursion
    \begin{align}
      c_0 &= 0, \\
      d_0 &= 1, \\
      c_k &= (\eta - 1)c_{k - 1} + d_{k - 1} \text{ for } k \geq 1, \\
      d_k &= \eta d_{k - 1} \text{ for } k \geq 1.
    \end{align}
    By induction, it is straightforward to show that these coefficients for $k \geq 1$ have the closed forms
    \begin{align}
      c_k &= \sum_{j = 0}^{k - 1}(-1)^{k - 1 - j} \eta^{j}, \\
      d_k &= \sum_{j = 1}^{k - 1}(-1)^{k - 1 - j} \eta^{j} = c_k + (-1)^{k}.
    \end{align}

    Conducting the expansion of $\widetilde{\mu}_k$ naively, we may write
    \begin{align*}
      \widetilde{\mu}_k
      &= \Tr(((\alpha \bm I_N + \widetilde{\bA})\beta^{-1}\widetilde{\bB})^{k}) \\
      &= \sum_{S \subseteq [k]} \alpha^{|S|} \Tr(\widetilde{\bA}^{\One\{1 \notin S\}}\beta^{-1}\widetilde{\bB} \cdots \widetilde{\bA}^{\One\{k \notin S\}} \beta^{-1}\widetilde{\bB}) \\
      &\equalscolon \sum_{S \subseteq [k]} f(S). \numberthis
    \end{align*}
    By the same token, the partial sums of $f(S)$ are given by
    \begin{equation}
        g(T) \colonequals \sum_{S \supseteq T} f(S) = \alpha^{|T|} \Tr(\bA^{\One\{1 \notin T\}}\beta^{-1}\widetilde{\bB} \cdots \bA^{\One\{k \notin T\}} \beta^{-1}\widetilde{\bB})
    \end{equation}
    upon substituting $\alpha \bm I_N + \widetilde{\bA}$ and expanding.
    The quantity we are interested in in this notation is $\widetilde{\mu}_k = g(\emptyset)$.
    By inclusion-exclusion (or, what is the same, \Mobius\ inversion in the poset of subsets of $[k]$), we have
    \begin{equation}
        f(T) = \sum_{S \supseteq T} (-1)^{|S| - |T|} g(S).
    \end{equation}
    We then have
    \begin{align*}
      \widetilde{\mu}_k
      &= g(\emptyset) \\
      &= \sum_{S \subseteq [k]} f(S) \\
      &= f(\emptyset) + \sum_{\substack{T \subseteq [k] \\ T \neq \emptyset}} \sum_{S \supseteq T} (-1)^{|T| - |S|} g(S)\\
      &= f(\emptyset) + \sum_{\substack{S \subseteq [k] \\ S \neq \emptyset}} g(S) \sum_{\substack{T \subseteq S \\ T \neq \emptyset}} (-1)^{|T| - |S|}
      \intertext{and by another inclusion-exclusion calculation}
      &= f(\emptyset) + \sum_{\substack{S \subseteq [k] \\ S \neq \emptyset}} (-1)^{|S| + 1} g(S) \\
      &= f(\emptyset) + (-1)^{k + 1} g([k]) + \sum_{\substack{S \subseteq [k] \\ 1 \leq |S| < k}} (-1)^{|S| + 1} g(S) \\
      &= \beta^{-k}\Tr((\widetilde{\bA} \widetilde{\bB})^k) + (-1)^{k + 1} \alpha^k \Tr((\beta^{-1} \widetilde{\bB})^k) \\
      &\hspace{2.525cm} + \sum_{\substack{S \subseteq [k] \\ 1 \leq |S| < k}} (-1)^{k - |S| + 1} \alpha^{k - |S|} \Tr(\bA^{\One\{1 \in S\}}\beta^{-1}\widetilde{\bB} \cdots \bA^{\One\{k \in S\}} \beta^{-1}\widetilde{\bB})
      \intertext{and finally, using that $(\beta^{-1}\widetilde{\bB})^k = c_k\beta^{-1}\widetilde{\bB} + d_k \bm I_N$,}
      &= (-1)^{k + 1} \alpha^k d_k N - \sum_{\substack{S \subseteq [k] \\ 1 \leq |S| < k}} (-\alpha)^{k - |S|} \Tr(\bA^{\One\{1 \in S\}}\beta^{-1}\widetilde{\bB} \cdots \bA^{\One\{k \in S\}} \beta^{-1}\widetilde{\bB}) \\
      &\hspace{2.525cm}+ (-1)^{k + 1} \frac{\alpha^k}{\beta}c_k\left(\frac{1}{N}\Tr(\bB) - \beta \right)N + \beta^{-k}\Tr((\widetilde{\bA} \widetilde{\bB})^k) \numberthis
\end{align*}

Our goal now is to show how the remaining traces above are linear combinations of the $\widetilde{\mu}_{\ell}$ with $\ell < k$, thus finding a recursion satisfied by the $\widetilde{\mu}_k$.
To this end, suppose we have a nonempty $S \subseteq [k]$ with $|S| = j$ and $S = \{s_1 < \cdots < s_j\}$.
Per Definition~\ref{def:set-partition}, we view $S$ as inducing an ordered partition of $k$ into $j$ parts, $p_1 + \cdots + p_j = k$, where $p_1 = s_2 - s_1, \dots, p_{j - 1} = s_j - s_{j - 1}, p_j = k + s_1 - s_j$.

Using this terminology, our expansion of powers of $\beta^{-1}\widetilde{\bB}$, and the idempotence of $\bA$, we may expand the traces remaining above as
\begin{align*}
  &\hspace{-1cm}\Tr(\bA^{\One\{1 \in S\}}\beta^{-1}\widetilde{\bB} \cdots \bA^{\One\{k \in S\}} \beta^{-1}\widetilde{\bB}) \\
  &= \Tr(\bA(\beta^{-1}\widetilde{\bB})^{p_1(S)} \cdots \bA (\beta^{-1}\widetilde{\bB})^{p_{|S|}(S)}) \\
  &= \Tr(\bA(c_{p_1(S)} \beta^{-1}\widetilde{\bB} + d_{p_1(S)} \bm I_N) \cdots \bA(c_{p_{|S|}(S)} \beta^{-1}\widetilde{\bB} + d_{p_{|S|}(S)} \bm I_N)) \\
  &= \sum_{R \subseteq [|S|]} \prod_{i \in R} c_{p_i(S)} \prod_{i \in [|S|] \setminus R} d_{p_i(S)} \cdot \Tr((\bA \beta^{-1}\widetilde{\bB})^{|R|}\bA) \\
  &= \Tr(\bA)\prod_{i \in [|S|]} d_{p_i(S)} + \sum_{r = 1}^{|S|} \widetilde{\mu}_r \sum_{R \in \binom{[|S|]}{r}}\prod_{i \in R} c_{p_i(S)} \prod_{i \in [|S|] \setminus R} d_{p_i(S)}. \numberthis
\end{align*}

The main remaining observation is that by the definition of the $q_{k, j, a}$ in Definition~\ref{def:q-poly}, we have
\begin{equation}
    \sum_{S \in \binom{[k]}{j}} \sum_{R \in \binom{[|S|]}{r}} \prod_{i \in R}c_{p_i(S)} \prod_{i \in [|S|] \setminus R} d_{p_i(S)} = q_{k, j, r}(\eta).
\end{equation}
Substituting into the previous step and using this, we find
\begin{align*}
  \widetilde{\mu}_k
  &= (-1)^{k + 1} \alpha^k d_k N \\
  &\hspace{0.5cm}- \sum_{\substack{S \subseteq [k] \\ 1 \leq |S| < k}} (-\alpha)^{k - |S|} \left(\Tr(\bA)\prod_{i \in [|S|]} d_{p_i(S)} + \sum_{r = 1}^{|S|} \widetilde{\mu}_r \sum_{R \in \binom{[|S|]}{r}}\prod_{i \in R} c_{p_i(S)} \prod_{i \in [|S|] \setminus R} d_{p_i(S)}\right) \\
  &\hspace{0.5cm}+ (-1)^{k + 1} \frac{\alpha^k}{\beta}c_k\left(\frac{1}{N}\Tr(\bB) - \beta \right)N + \Tr((\widetilde{\bA} \widetilde{\bB})^k) \\
  &= (-1)^{k + 1} \alpha^k d_k N - \Tr(\bA)\sum_{j = 1}^{k - 1} (-\alpha)^{k - j} q_{k, j, 0}(\eta) - \sum_{a = 1}^{k - 1} \widetilde{\mu}_a \sum_{j = 1}^{k - 1} (-\alpha)^{k - j} q_{k, j, a}(\eta) \\
  &\hspace{0.5cm}+ (-1)^{k + 1} \frac{\alpha^k}{\beta}c_k\left(\frac{1}{N}\Tr(\bB) - \beta \right)N + \beta^{-k}\Tr((\widetilde{\bA} \widetilde{\bB})^k), \numberthis
\end{align*}
and substituting for $\eta$ and the remaining occurrences of $c_k$ and $d_k$ in terms of $\beta$ gives the result as stated.
\end{proof}

\subsection{MANOVA moment recursion}

Not surprisingly, the ``main terms'' of the recursion of Lemma~\ref{lem:necklace-expansion} describe a recursion satisfied by the moments of $\MANOVA(\alpha, \beta)$.
We state this formally below.

\begin{lemma}
    \label{lem:manova-recursion}
    Write $m_k \colonequals \EE_{X \sim \MANOVA(\alpha, \beta)}[X^k]$.
    Then, for $k \geq 1$, we have
    \begin{equation}
        m_k = \beta^k \sum_{\ell = 0}^k \binom{k}{\ell}\widetilde{m}_k,
    \end{equation}
    where the $\widetilde{m}_k$ satisfy the recursion
    \begin{align}
      \widetilde{m}_0 &= \alpha, \\
      \widetilde{m}_k
      &= \alpha^{k}\left(\sum_{b = 1}^{k - 1}(1 - \beta^{-1})^b\right) - \sum_{a = 0}^{k - 1} \widetilde{m}_{a} \sum_{j = 1}^{k - 1} (-\alpha)^{k - j} q_{k, j, a}(\beta^{-1} - 1) \text{ for } k \geq 1.
    \end{align}
\end{lemma}

\begin{remark}[MANOVA moment formulas]
There is a more explicit formula known for these moments in terms of the Narayana polynomials (see, e.g., \cite{DE-2015-InfiniteRandomMatrixTheory}).
The prior work \cite{MMP-2019-RandomSubensembles}, which proved the $\beta = \frac{1}{2}$ case of Theorem~\ref{thm:untf} for ETFs, worked directly with this formula.
This formula, while a closed-form expression, seems rather difficult to use, while the recursion above, being more transparently a limit of the recursion in Lemma~\ref{lem:necklace-expansion}, makes it easier to compare the empirical and limiting moments.
We also have not been able to produce a direct combinatorial proof that the closed formula for these moments satisfies the recursion above.
We do, however, give a proof of this fact using generating functions and orthogonal polynomials in Appendix~\ref{app:gf}.
\end{remark}

\begin{proof}
    By Theorem~\ref{thm:invariant-model}, if $\bA$ and $\bB$ are orthogonal projections to uniformly random subspaces of $\CC^N$ of dimension $\alpha N$ and $\beta N$, respectively, then
    \begin{equation}
        m_k = \lim_{N \to \infty} \frac{1}{N}\EE \Tr(\bA\bB\bA)^k.
    \end{equation}
    The result then follows from Lemma~\ref{lem:necklace-expansion}, since, by Theorem~\ref{thm:invariant-freeness}, $\bA$ and $\bB$ are asymptotically free, whereby
    \begin{equation}
        \lim_{N \to \infty} \frac{1}{N}\EE \Tr((\bA - \alpha \bm I_N)(\bB - \beta \bm I_N))^k = 0,
    \end{equation}
    and the other non-leading order terms are easily shown to also converge to zero.
\end{proof}

\subsection{Error term recursion}

Finally, we may combine the results of the two preceding sections to give a description of the error term by which an empirical moment deviates from the corresponding MANOVA moment.
This will be especially useful for our study of the largest eigenvalue of $\bA\bB\bA$.

\begin{lemma}
    \label{lem:error-term-recursion}
    We have
    \begin{equation}
        \Tr(\bA\bB\bA)^k = \Ex_{X \sim \MANOVA(\alpha, \beta)}[X^k] \cdot N + \Delta_k,
    \end{equation}
    where
    \begin{equation}
        \Delta_k = \beta^k \sum_{\ell = 0}^k \binom{k}{\ell} \widetilde{\Delta}_{\ell},
    \end{equation}
    where the $\widetilde{\Delta}_{\ell}$ satisfy the recursion
    \begin{align*}
      \widetilde{\Delta}_0 &= \Tr(\bA) - \alpha N, \numberthis \\
      \widetilde{\Delta}_k &= -\sum_{a = 0}^{k - 1} \widetilde{\Delta}_a \sum_{j = 1}^{k - 1} (-\alpha)^{k - j} q_{k, j, a}(\beta^{-1} - 1) \\
                           &\hspace{2cm} + (-1)^{k + 1} \frac{\alpha^k}{\beta}\left(\sum_{b = 0}^{k - 1}(1 - \beta^{-1})^b\right)\left(\Tr(\bB) - \beta N \right) \\
                           &\hspace{2cm} + \frac{1}{\beta^k}\Tr\big((\bA - \alpha \bm I_N)(\bB - \beta \bm I_N)\big)^k \text{ for } k \geq 1. \numberthis
    \end{align*}
\end{lemma}
\begin{proof}
    The result follows simply from subtracting the recursion of Lemma~\ref{lem:manova-recursion} from that of Lemma~\ref{lem:necklace-expansion}.
\end{proof}

\section{Maximum eigenvalue in Kesten-McKay case}
\label{sec:max-eval-km}

\subsection{Sufficient condition: Proof of Theorem~\ref{thm:edge-km}}

Recall that our goal is to show that $\lambda_{\max}(\bA\bB\bA)$ converges in probability to $\edge(\alpha, \beta)$.
We begin with some simple reductions of the original statement.
We have assumed that one of $\alpha$ or $\beta$ equals $\frac{1}{2}$, but since we may exchange the roles of $\bA$ and $\bB$ without changing $\lambda_{\max}(\bA\bB\bA) = \|(\bA\bB)(\bB\bA)\|$, let us assume without loss of generality that $\beta = \frac{1}{2}$.

Recall that we assume the e.s.d.\ of $\bA\bB\bA$ converges in probability to $\MANOVA(\alpha, \beta = \frac{1}{2})$.
So, by the definition of support, we must have for any $\epsilon > 0$ that $\lambda_{\max}(\bA\bB\bA) \geq \edge(\alpha, \frac{1}{2}) - \epsilon$ with high probability.
Thus, it suffices to show the upper bound that, again for any $\epsilon > 0$, with high probability $\lambda_{\max}(\bA\bB\bA) \leq \edge(\alpha, \frac{1}{2}) + \epsilon$.

When $\alpha \geq \frac{1}{2}$, then $\edge(\frac{1}{2},\alpha) = 1$.
We have $\|\bA\bB\bA\| \leq \|\bA\|^2\|\bB\| = 1$, so in this case the upper bound follows trivially and with probability 1.
So, we may further assume that $\alpha < \frac{1}{2}$.
In this case,
\begin{equation}
    \edge(\alpha, \beta) = r_+(\alpha, \beta) = \alpha + \beta - 2\alpha\beta + 2\sqrt{\alpha(1 - \alpha)\beta(1 - \beta)} = \frac{1}{2} + \sqrt{\alpha(1 - \alpha)}.
\end{equation}

Let us next establish what we expect from the error terms in Lemma~\ref{lem:error-term-recursion} if we are to use the Lemma to prove our result.
It will suffice to show that $|\Delta_k| \leq \edge(\frac{1}{2}, \alpha)^k \exp(o(k))$.
Now, if we have an exponential growth bound $|\widetilde{\Delta}_k| \leq c^k \exp(o(k))$, then the Lemma implies that $|\Delta_k| \leq (\frac{1}{2})^k(1 + c)^k\exp(o(k)) = (\frac{1}{2} + \frac{1}{2}c)^k\exp(o(k))$.
Thus, our attempt will succeed so long as $\frac{1}{2} + \frac{1}{2}c \leq \edge(\frac{1}{2}, \alpha) = \frac{1}{2} + \sqrt{\alpha(1 - \alpha)}$, or equivalently so long as $c \leq 2\sqrt{\alpha(1 - \alpha)}$.
Our proof will proceed by solving the recursion for the $\widetilde{\Delta}_k$ and verifying that such a bound does hold.

It is in solving the recursion, to which we turn our attention now, that we will rely on our assumption that $\beta = \frac{1}{2}$.
We first show that in this case, the linear combination of $q_{k, j, a}(\beta^{-1} - 1) = q_{k, j, a}(1)$ appearing in the recursion has a concise combinatorial description.

\begin{definition}[Modified Lucas triangle]
    For $n \geq 1$ and $0 \leq j \leq n$, we define
    \begin{equation}
        T(n, j) \colonequals \binom{n - 1}{j} + 2 \binom{n - 1}{j - 1} = \binom{n}{j} + \binom{n - 1}{j - 1}.
    \end{equation}
    We also define $T(0, 0) \colonequals 1$, though we emphasize that this is \emph{not} the standard extension to the case $n = 0$, which is $T(0, 0) = 2$.
\end{definition}

\begin{proposition}
    \label{prop:q-sum-km}
    For all $k \geq 1$ and $0 \leq a \leq k$,
    \begin{equation}
        \sum_{j = 1}^{k} (-\alpha)^{k - j} q_{k, j, a}(1) = \left\{\begin{array}{ll} T(\frac{k}{2}, \frac{k}{2}) (\alpha(\alpha - 1))^{(k - a) / 2} - 2\alpha^k & \text{if } 0 = a \equiv k \Mod{2}, \\ T(\frac{k + a}{2}, \frac{k - a}{2}) (\alpha(\alpha - 1))^{(k - a) / 2} & \text{if } 0 < a \equiv k \Mod{2}, \\ 0 & \text{otherwise}. \end{array}\right.
    \end{equation}
\end{proposition}
\begin{proof}
    We give a generating function proof using the tools developed in Appendix~\ref{app:gf}.
    By Proposition~\ref{prop:q-gf}, and using the notation of that result, we have the generating function
    \begin{align*}
      \sum_{k \geq 1} \sum_{a = 0}^k \left(\sum_{j = 1}^{k} (-\alpha)^{k - j} q_{k, j, a}(1)\right)x^k z^a
      &= \sum_{k \geq 1} \sum_{j = 0}^{k} \sum_{a = 0}^k q_{k, j, a}(1) (-\alpha x)^k \left(-\frac{1}{\alpha}\right)^j z^a \\
      &= Q\left(1, -\alpha x, -\frac{1}{\alpha}, z\right) \\
      &= \frac{-2\alpha x^2 + (\alpha^2 x^3 + x)z}{(\alpha^4 - \alpha^3)x^4 - (2\alpha^2 - \alpha)x^2 + (\alpha^2x^3 - x)z + 1}.
    \end{align*}
    On the other hand, we may write the generating function of the right-hand side as
    \begin{align*}
      &\sum_{k \geq 1} \sum_{a = 0}^k \left\{\begin{array}{ll} T(\frac{k}{2}, \frac{k}{2}) (\alpha(\alpha - 1))^{(k - a) / 2} - 2\alpha^k & \text{if } 0 = a \equiv k \Mod{2}, \\ T(\frac{k + a}{2}, \frac{k - a}{2}) (\alpha(\alpha - 1))^{(k - a) / 2} & \text{if } 0 < a \equiv k \Mod{2}, \\ 0 & \text{otherwise} \end{array}\right\}x^k z^a \\
      &= \sum_{k \geq 1} \sum_{\substack{0 \leq a \leq k \\ a \equiv k \Mod{2}}} T\left(\frac{k + a}{2}, \frac{k - a}{2}\right) (\alpha(\alpha - 1))^{(k - a) / 2}x^kz^a - 2\sum_{k \geq 1} \alpha^{2k} x^{2k} \\
      &= \sum_{k \geq 1} \sum_{a = 0}^k T\left(k, a\right) (zx)^{k}\left(\frac{\alpha(\alpha - 1)x}{z}\right)^{a} - \frac{2\alpha^2 x^2}{1 - \alpha^2 x^2}
        \intertext{and, separating out the $a = 0$ terms for which $T(k, 0) = 1$ and substituting for the remaining terms the definition of $T(k, a)$,}
                                                                                                                                          &= \alpha(\alpha - 1)x^2 \sum_{k \geq 0}\sum_{a = 0}^k \binom{k}{a} (zx)^{k}\left(\frac{\alpha(\alpha - 1)x}{z}\right)^{a} + \sum_{k \geq 1}\sum_{a = 1}^k \binom{k}{a} (zx)^{k}\left(\frac{\alpha(\alpha - 1)x}{z}\right)^{a} \\
      &\hspace{1cm} + \sum_{k \geq 1} (zx)^k - \frac{2\alpha^2 x^2}{1 - \alpha^2 x^2} \\
      &= \alpha(\alpha - 1)x^2 \sum_{k \geq 0} (zx)^k \left(1 + \frac{\alpha(\alpha - 1)x}{z}\right)^k + \sum_{k \geq 1} (zx)^k \left(\left(1 + \frac{\alpha(\alpha - 1)x}{z}\right)^k - 1\right) \\
      &\hspace{1cm} + \sum_{k \geq 1} (zx)^k - \frac{2\alpha^2 x^2}{1 - \alpha^2 x^2} \\
      &= \frac{zx + 2\alpha(\alpha - 1)x^2}{1 - zx - \alpha(\alpha - 1)x^2} - \frac{2\alpha^2 x^2}{1 - \alpha^2 x^2}, \numberthis
    \end{align*}
    and it is straightforward to verify the remaining algebraic equality.
\end{proof}

We now use this observation to solve the recursion.
The solution is strikingly simple thanks to the connection between $q_{k, j, a}(1)$ and the Lucas triangle.
\begin{lemma}
    \label{lem:tDelta-tdelta}
    In the setting of Lemma~\ref{lem:error-term-recursion} with $\beta = \frac{1}{2}$, define
    \begin{align*}
      \widetilde{\delta}_0 &\colonequals \Tr(\bA - \alpha \bm I_N) = \widetilde{\Delta}_0, \numberthis \\
      \widetilde{\delta}_k
                           &\colonequals \Tr\big((\bA - \alpha \bm I_N)(2\bB - \bm I_N)\big)^k + \One\{k \equiv 0 \Mod{2}\} \, 2\alpha^k \left(\Tr(\bA) - \alpha N \right) \\
      &\hspace{5.6cm} + \One\{k \equiv 1 \Mod{2}\} \, 2\alpha^k \left(\Tr(\bB) - \frac{1}{2}N \right). \numberthis
    \end{align*}
    Then, for all $k \geq 1$,
    \begin{equation}
        \widetilde{\Delta}_k = \sum_{\substack{0 \leq a \leq k \\ a \equiv k \Mod{2}}} \binom{k}{\frac{k + a}{2}} (\alpha(1 - \alpha))^{\frac{k - a}{2}} \widetilde{\delta}_a.
    \end{equation}
\end{lemma}
\noindent
This result will follow fairly directly from the following combinatorial fact.
We defer its proof, which uses results from the theory of Riordan arrays and their generating functions, to Appendix~\ref{app:riordan}.
We emphasize that, using these tools, the proof of this result is a straightforward generating function calculation.
The result also appears not to be original; its statement is mentioned without proof in the Online Encyclopedia of Integer Sequences as the relationship between sequences \href{https://oeis.org/A156290}{[A156290]} and \href{https://oeis.org/A156308}{[A156308]} and sequences \href{https://oeis.org/A113187}{[A113187]} and \href{https://oeis.org/A111125}{[A111125]}.
\begin{proposition}
    \label{prop:riordan-inv}
    Let $\bX, \bY$ denote infinite lower-triangular matrices with entries given by, for $k, a \geq 0$,
    \begin{align}
      X_{ka} &= \One\{a \leq k\} \, \One\{a \equiv k \Mod{2}\}\, T\left(\frac{k + a}{2}, \frac{k - a}{2}\right), \\
      Y_{ka} &= \One\{a \leq k\} \, \One\{a \equiv k \Mod{2}\}\, (-1)^{(k + a) / 2}\binom{k}{ \frac{k + a}{2}}.
    \end{align}
    These matrices are inverses, in the sense that $\bX\bY = \bY\bX$ is the infinite identity matrix; in other words, for all $k, \ell \geq 0$,
    \begin{equation}
        \sum_{a = 0}^{\infty} X_{ka} Y_{a\ell} = \sum_{a = 0}^{\infty} Y_{\ell a} X_{ak} = \One\{k = \ell\},
    \end{equation}
    where we emphasize that this statement only involves finite summations since all but finitely many terms are zero.
\end{proposition}

\begin{proof}[Proof of Lemma~\ref{lem:tDelta-tdelta}]
    The recursion of Lemma~\ref{lem:error-term-recursion} may be rewritten, using that $\beta^{-1} = 2$ and applying Proposition~\ref{prop:q-sum-km}, as, for $k \geq 1$,
    \begin{align*}
      \widetilde{\Delta}_k &= -\sum_{a = 0}^{k - 1} \widetilde{\Delta}_a \sum_{j = 1}^{k - 1} (-\alpha)^{k - j} q_{k, j, a}(1) \\
                           &\hspace{2cm} + \One\{k \equiv 1 \Mod{2}\} \, \alpha^k\left(\Tr(\beta^{-1}\bB) - N \right) \\
                           &\hspace{2cm} + \Tr\big((\bA - \alpha \bm I_N)(\beta^{-1}\bB - \bm I_N)\big)^k \\
                           &= -\sum_{\substack{0 \leq a \leq k - 1 \\ a \equiv k \Mod{2}}} T\left(\frac{k + a}{2}, \frac{k - a}{2}\right) \widetilde{\Delta}_a (\alpha(\alpha - 1))^{(k - a) / 2} \\
                           &\hspace{2cm} + \One\{k \equiv 0 \Mod{2}\} \, 2\alpha^k(\underbrace{\Tr(\bA) - \alpha N}_{= \widetilde{\Delta}_0}) \\
                           &\hspace{2cm} + \One\{k \equiv 1 \Mod{2}\} \, \alpha^k\left(\Tr(\beta^{-1}\bB) - N \right) \\
                           &\hspace{2cm} + \Tr\big((\bA - \alpha \bm I_N)(\beta^{-1}\bB - \bm I_N)\big)^k. \numberthis
    \end{align*}
    The last three terms here are precisely $\widetilde{\delta}_k$.
    Since $T(k, 0) = 1$ for all $k$, we may rewrite again as an implicit recursion,
    \begin{equation}
        \label{eq:delta-Delta}
        (\alpha(\alpha - 1))^{-(k + e(k))/2} \widetilde{\delta}_k = \sum_{\substack{0 \leq a \leq k \\ a \equiv k \Mod{2}}} T\left(\frac{k + a}{2}, \frac{k - a}{2}\right) \widetilde{\Delta}_a (\alpha(\alpha - 1))^{-(a + e(a)) / 2},
    \end{equation}
    where $e(k) \colonequals \One\{k \equiv 1 \Mod{2}\}$ and adding this ensures that no fractional powers are taken no matter the parity of $k$.
    We note also that this equation extends to $k = 0$, by the combination of our definitions of $\widetilde{\delta}_0 \colonequals \widetilde{\Delta}_0$ and $T(0, 0) \colonequals 1$ (this is the reason for the unconventional latter definition).

    The result will now follow from computing the inverse of the infinite lower-triangular matrix formed by the $(\One\{a \equiv k \Mod{2}\}\, T(\frac{k + a}{2}, \frac{k - a}{2}))_{0 \leq a \leq k}$, where $k$ indexes the rows and $a$ indexes the columns.
    By Proposition~\ref{prop:riordan-inv}, this inverse is given by the infinite lower-triangular matrix formed by $(\One\{a \equiv k \Mod{2}\}\, (-1)^{(k + a) / 2} \binom{k}{(k + a) / 2})_{0 \leq a \leq k}$.
    Thus, inverting \eqref{eq:delta-Delta} gives
    \begin{equation}
        \widetilde{\Delta}_k (\alpha(\alpha - 1))^{-(k + e(k)) / 2} = \sum_{\substack{0 \leq a \leq k \\ a \equiv k \Mod{2}}} (-1)^{(k + a) / 2} \binom{k}{\frac{k + a}{2}} (\alpha(\alpha - 1))^{-(a + e(a)) / 2}\widetilde{\delta}_a,
    \end{equation}
    and the result follows upon rearranging.
\end{proof}

We now proceed to complete the main proof.
\begin{proof}[Proof of Theorem~\ref{thm:edge-km}]
    Note that
    \begin{equation}
        \max\left\{|\Tr(\bA) - \alpha N|, \left| \Tr(\bB) - \frac{1}{2}N\right|\right\} \lesssim N = \exp(o(k)).
    \end{equation}
    So, we have
    \begin{align*}
      |\widetilde{\delta}_k|
      &\leq \left|\Tr((\bA - \alpha \bm I_N)(2\bB - \bm I_N))^k\right| + \alpha^k \exp(o(k)) \\
      &= (\sqrt{\alpha(1 - \alpha)})^{k}\left|\Tr(\what{\bA}\what{\bB})^k\right| + \alpha^k \exp(o(k)) \\
      &\leq \left((\sqrt{\alpha(1 - \alpha)})^{k} + \alpha^k\right) \exp(o(k)),
        \intertext{and since $\alpha < \frac{1}{2}$ we have $\alpha < \sqrt{\alpha(1 - \alpha)}$, so we simply have}
      &\leq (\sqrt{\alpha(1 - \alpha)})^{k} \exp(o(k)). \numberthis
    \end{align*}
    Then, by Lemma~\ref{lem:tDelta-tdelta},
    \begin{align*}
      |\widetilde{\Delta}_k|
      &\leq \exp(o(k)) \sum_{\substack{0 \leq a \leq k \\ a \equiv k \Mod{2}}} \binom{k}{\frac{k + a}{2}} (\alpha(1 - \alpha))^{\frac{k - a}{2}} (\alpha(1 - \alpha))^{a/2} \\
      &= (\sqrt{\alpha(1 - \alpha)})^{k} \exp(o(k)) \sum_{\substack{0 \leq a \leq k \\ a \equiv k \Mod{2}}} \binom{k}{\frac{k + a}{2}} \\
      &\leq (2\sqrt{\alpha(1 - \alpha)})^{k} \exp(o(k)). \numberthis
    \end{align*}
    Note that this is precisely the bound on $|\widetilde{\Delta}_k|$ that we suggested would be sufficient at the beginning of this section.
    Finally, by Lemma~\ref{lem:error-term-recursion},
    \begin{align*}
      |\Delta_k|
      &\leq \frac{1}{2^k} \sum_{\ell = 0}^k \binom{k}{\ell} |\widetilde{\Delta}_{\ell}| \\
      &\leq \left(\frac{1}{2}(1 + 2\sqrt{\alpha(1 - \alpha)})\right)^k \exp(o(k)) \\
      &= \edge\left(\frac{1}{2}, \alpha\right)^k \exp(o(k)). \numberthis
    \end{align*}
    Now, since $N = \exp(o(k))$ by assumption, we have
    \begin{align}
      |\EE\Tr(\bA\bB\bA)^k| \leq \left|\Ex_{X \sim \MANOVA(\alpha, \beta)}[X^k]\right| \cdot N + |\Delta_k| \leq \edge\left(\frac{1}{2}, \alpha\right)^k \exp(o(k)).
    \end{align}
    We now apply a standard argument to complete the proof.
    Suppose $\epsilon > 0$.
    Since $\bA\bB\bA$ is positive semidefinite, and using Markov's inequality, we have
    \begin{align*}
      \PP\left[\lambda_{\max}(\bA\bB\bA) \geq \edge\left(\frac{1}{2}, \alpha\right) + \epsilon\right]
      &\leq \PP\left[\Tr(\bA\bB\bA)^k \geq \left(\edge\left(\frac{1}{2}, \alpha\right) + \epsilon\right)^k\right] \\
      &\leq \EE \Tr(\bA\bB\bA)^k \cdot \left(\edge\left(\frac{1}{2},\alpha\right) + \epsilon\right)^{-k} \\
      &\leq \left(1 + \frac{\epsilon}{\edge\left(\frac{1}{2}, \alpha\right)}\right)^{-k} \exp(o(k)), \numberthis
    \end{align*}
    which tends to zero as $\iota \to \infty$, since $k = k(\iota)$ is diverging.
\end{proof}

\begin{remark}[The general MANOVA case]
    \label{rem:max-eval-manova}
    It seems likely that Theorem~\ref{thm:edge-km} should generalize to the case of arbitrary $\alpha, \beta \in (0, 1)$.
    The main missing ingredient in such an argument would be a generalization to this case of Lemma~\ref{lem:tDelta-tdelta}.
    We do not know whether the summations treated by Proposition~\ref{prop:q-sum-km} still have some relation to Riordan arrays when $\beta \neq \frac{1}{2}$; this would be an interesting topic for further investigation.
\end{remark}

\subsection{Application to invariant model: Proof of Theorem~\ref{thm:edge-km-unitary}}

We first gather some preliminaries about the joint moments of entries of Haar-distributed unitary matrices.
The following is an important manifestation of the Schur-Weyl duality of representation theory.

\begin{theorem}[Corollary 2.4 of \cite{CS-2006-HaarMeasureMoments}]
    \label{thm:weingarten}
    Let $\bU \in \sU(N)$ be distributed according to Haar measure.
    There is a function $W = W_k: S_k \to \RR$, called the \emph{Weingarten function}, such that, for any $\bm i, \bm j \in [N]^k$ and $\bm i^{\prime}, \bm j^{\prime} \in [N]^{k^{\prime}}$, we have
    \begin{equation}
        \EE[U_{i_1j_1} \cdots U_{i_k j_k} \overline{U_{i_1^{\prime} j_1^{\prime}}} \cdots \overline{U_{i_{k^{\prime}}^{\prime} j_{k^{\prime}}^{\prime}}}] = \sum_{\substack{\sigma \in S_k \\ \sigma \bm i = \bm i^{\prime}}} \sum_{\substack{\tau \in S_k \\ \tau \bm j = \bm j^{\prime}}} W(\sigma\tau^{-1}).
    \end{equation}
\end{theorem}

The bound we use on the Weingarten function will be given in terms of the following measure of the ``size'' of a permutation.
\begin{definition}[Cayley weight]
    The \emph{Cayley weight} of a permutation $\sigma \in S_k$, denoted $|\sigma|$, is the minimum number of transpositions (2-cycles) that can be composed to form $\sigma$.
    For $\sigma, \tau \in S_k$, $|\sigma^{-1}\tau|$ is called the \emph{transposition distance} between $\sigma$ and $\tau$.
\end{definition}
\noindent
We recall some basic and well-known properties of these notions.
First, for any $\sigma$,
\begin{equation}
    \label{eq:transp-distance}
    |\sigma| = k - |\cyc(\sigma)|.
\end{equation}
Also, the transposition distance gives a metric on $S_k$, in particular satisfying the triangle inequality.
It may be interpreted as the minimum number of transpositions by which $\sigma$ must be multiplied to reach $\tau$.

The following crucial tool controls the Weingarten function with non-asymptotic bounds in a way depending only on the Cayley weight.
\begin{theorem}[Theorem 1.1 and Lemma 3.3 of \cite{CM-2017-WeingartenCalculusBounds}]
    \label{thm:weingarten-bound}
    If $N > \sqrt{6} k^{7/4}$, then, for any $\sigma \in S_k$ and for $W_N$ the function in Theorem~\ref{thm:weingarten},
    \begin{equation}
        |W_N(\sigma)| \leq \frac{1}{1 - \frac{6k^{7/2}}{N^2}} N^{-k - |\sigma|} \prod_{C \in \cyc(\sigma)} \mathsf{Cat}(|C| - 1) \leq \frac{1}{1 - \frac{6k^{7/2}}{N^2}} N^{-k} \left(\frac{4}{N}\right)^{|\sigma|},
    \end{equation}
    where $\mathsf{Cat}(n) = \frac{1}{n + 1} \binom{2n}{n}$ are the Catalan numbers.
\end{theorem}
\noindent
The second estimate above follows from the inequality $\mathsf{Cat}(k) \leq 4^k$ together with \eqref{eq:transp-distance}.

\begin{definition}[Tuple action of $S_k$ and compatibility]
    For $\sigma \in S_k$ and $\bm i \in [N]^k$ for some $N \geq 1$, we write $\sigma \bm i = (i_{\sigma(1)}, \dots, i_{\sigma(k)})$.
    We say that $\sigma$ and $\bm i$ are \emph{compatible} if $\sigma \bm i = \bm i$.
\end{definition}

The following quantity will arise often in the calculations to come, so we give a separate notation for it:
\begin{equation}
    w(\alpha) \colonequals \sqrt{\alpha(1 - \alpha)} \in (0, 1).
\end{equation}
Recall that this is the standard deviation of a random variable with law $\Ber(\alpha)$.

We now prove some preliminary bounds on quantities that will appear in our proof of Theorem~\ref{thm:edge-km-unitary}.
\begin{lemma}
    \label{lem:f-sigma-bound}
    Let $a_1, \dots, a_N \sim \Ber(\alpha)$ be i.i.d., and write $\what{a}_i \colonequals (a_i - \alpha) / w(\alpha)$.
    Suppose that $N \geq 16 k^4 / w(\alpha)^2$.
    Write
    \begin{equation}
        f_{\alpha}(\sigma) \colonequals \left| \sum_{\substack{\bm i \in [N]^k \\ \bm i \text{ compatible with } \sigma}} \EE[\what{a}_{i_1} \cdots \what{a}_{i_k}] \right|.
    \end{equation}
    Then, for any $\sigma \in S_k$,
    \begin{equation}
      f_{\alpha}(\sigma) \lesssim k\exp(k^{3/4}) \, w(\alpha)^{-\sum_{a = 3}^k (a - 2)|\cyc_a(\sigma)|} \, (|\cyc_1(\sigma)|)!! \, N^{|\cyc_1(\sigma)| / 2 + |\cyc_{\geq 2}(\sigma)|}.
    \end{equation}
\end{lemma}
\begin{proof}
    First, we bound the moments of the $\what{a}_i$.
    We have $\EE \, \what{a}_i = 0$, and for all $\ell \geq 2$,
    \begin{align*}
      \EE \, \what{a}_i^{\ell}
      &= \alpha \left(\frac{1 - \alpha}{\sqrt{\alpha(1 - \alpha)}}\right)^{\ell} + (1 - \alpha)\left(\frac{-\alpha}{\sqrt{\alpha(1 - \alpha)}}\right)^{\ell} \\
      &\leq \left(\frac{1}{\sqrt{\alpha(1 - \alpha)}}\right)^{\ell - 2}\left((1 - \alpha)^{\ell - 1} + \alpha^{\ell - 1}\right) \\
      &\leq \left(\frac{1}{\sqrt{\alpha(1 - \alpha)}}\right)^{\ell - 2} \\
      &= w^{-\ell + 2}. \numberthis
    \end{align*}

    Note that $\bm i$ is compatible with $\sigma$ if and only if it is constant on the cycles of $\sigma$.
    Let $\pi$ be the partition of $[k]$ formed by the cycles of $\sigma$.
    Then, we may rewrite in terms of the partition $\rho$ of $[k]$ on which $\bm i$ is constant as
    \begin{align*}
      f_{\alpha}(\sigma)
      &\leq \sum_{\substack{\rho \in \Part([k]) \\ \rho \geq \pi}} N^{|\rho|} \prod_{S \in \rho} |\EE[\what{a}_1^{|S|}]| \\
      &\leq \sum_{\substack{\rho \in \Part([k]) \\ \rho \geq \pi \\ |S| \geq 2 \text{ for all } S \in \rho}} N^{|\rho|} \prod_{S \in \rho} w^{-|S| + 2} \\
      &= w^{-k} \sum_{\substack{\rho \in \Part([k]) \\ \rho \geq \pi \\ |S| \geq 2 \text{ for all } S \in \rho}} (w^2 N)^{|\rho|}
      \intertext{Let us rewrite this one more time: let $\ell$ be the number of fixed points of $\sigma$, and $m$ be the total number of cycles. Then, we have:}
      &= w^{-k} \sum_{\substack{\rho \in \Part([m]) \\ 1, \dots, \ell \text{ not singletons in } \rho}} (w^2 N)^{|\rho|}
      \intertext{Now, grouping those parts of $\rho$ consisting of only elements among $1, \dots, \ell$ and all the others, we have}
      &= w^{-k} \sum_{a = 0}^{\ell} \sum_{\rho \in \Part([m - \ell])} \sum_{\substack{\pi \in \Part([a]) \\ |S| \geq 2 \text{ for all } S \in \pi}} |\rho|^{\ell - a} (w^2N)^{|\pi| + |\rho|} \\
      &= w^{-k} \sum_{a = 0}^{\ell} \sum_{\rho \in \Part([m - \ell])} |\rho|^{\ell - a} (w^2N)^{|\rho|}\sum_{\substack{\pi \in \Part([a]) \\ |S| \geq 2 \text{ for all } S \in \pi}} (w^2 N)^{|\pi|}
      \intertext{Using Proposition~\ref{prop:S1-sum} on the inner sum and bounding $|\rho|^{\ell - a} \leq (m - \ell)^{\ell - a}$ on the outer sum,}
      &\leq w^{-k} \sum_{a = 0}^{\ell} \exp(a^{3/4}) \cdot a !! (w^2 N)^{a/2} (m - \ell)^{\ell - a} \sum_{\rho \in \Part([m - \ell])} (w^2N)^{|\rho|}
        \intertext{Using Proposition~\ref{prop:S-sum} on the remaining sum and using $m - \ell \leq m \leq k$,}
      &\leq \exp\left(\frac{k^2}{w^2 N} + k^{3/4}\right) w^{-k} \sum_{a = 0}^{\ell} a !! \, (w^2 N)^{m - \ell + a/2} (m - \ell)^{\ell - a}
        \intertext{Since $m - \ell \leq m \leq k \leq \sqrt{w^2 N}$, the largest summand is that with $a = \ell$, so we can bound}
      &\leq \exp\left(\frac{k^2}{w^2 N} + k^{3/4}\right) (\ell + 1) w^{-k} \ell !! \, (w^2N)^{m - \ell / 2}, \numberthis
    \end{align*}
    completing the proof.
\end{proof}

\begin{lemma}
    \label{lem:cycle-sum}
    For any $\sigma \in S_k$ and $x > 2k^2$,
    \begin{equation}
        \sum_{\rho \in S_k \text{ a } k\text{-cycle}} x^{-|\sigma \rho^{-1}|} \leq 2\frac{(k - 1)!}{(k - |\cyc(\sigma)| + 1)!} k^{|\cyc(\sigma)|} x^{-|\cyc(\sigma)| + 1}.
    \end{equation}
\end{lemma}
\begin{proof}
    Since composing with a transposition can only change the number of cycles of a permutation by one, for any $\rho$, we have $|\sigma\rho^{-1}| \geq d_{\min} \colonequals |\cyc(\sigma)| - 1$.
    On the other hand, for any $\rho$ (even one that is not a $k$-cycle), we always have $|\sigma \rho^{-1}| \leq k - 1$.
    So, we may write
    \begin{equation}
        \sum_{\rho \in S_k \text{ a } k\text{-cycle}} x^{-|\sigma \rho^{-1}|} = \sum_{d = d_{\min}}^{k - 1} x^{-d} \cdot \#\{\rho \in S_k \text{ a } k\text{-cycle}: |\sigma\rho^{-1}| = d\}.
    \end{equation}

    Let $\rho$ have $|\sigma \rho^{-1}| = |\rho \sigma^{-1}| = d$.
    Then, there must be transpositions $\tau_1^{(0)}, \dots, \tau_d^{(0)}$ such that $\rho = \tau_d^{(0)} \cdots \tau_1^{(0)} \sigma$.
    Let $i$ be the smallest index such that $\tau_i^{(0)}$ transposes two elements in different cycles of $\sigma$.
    Using that transpositions on disjoint pairs of indices commute and the identity $(i \, j)(j \, k) = (i \, j \, k) = (k \, i) (i \, j)$, we may instead write $\rho = \tau_d^{(1)} \cdots \tau_1^{(1)}\sigma$, where $\tau_1^{(1)} = \tau_i^{(0)}$.
    Repeating this procedure on the leftmost sequence of $d - 1$ transpositions in $\tau_d^{(1)} \cdots \tau_2^{(1)}(\tau_1^{(1)} \sigma)$, and performing this procedure a total of $d_{\min}$ times, we reach an expression $\rho = \tau_d^{(d_{\min})} \cdots \tau_{d_{\min} + 1}^{(d_{\min})} (\tau_{d_{\min}}^{(d_{\min})} \cdots \tau_{1}^{(d_{\min})} \sigma)$, where $\rho_0 \colonequals \tau_{d_{\min}}^{(d_{\min})} \cdots \tau_{1}^{(d_{\min})} \sigma$ is a $k$-cycle with $|\sigma \rho_0^{-1}| = d_{\min}$.
    Moreover, we have $|\rho \rho_0^{-1}| \leq d - d_{\min}$ by the above expression, and indeed we must have $|\rho \rho_0^{-1}| = d - d_{\min}$, since otherwise the triangle inequality for transposition distance would be violated.

    We have shown that, for any $k$-cycle $\rho$ at transposition distance $d$ from $\sigma$, there exists a $k$-cycle $\rho_0$ at transposition distance $d_{\min}$ from $\sigma$ that is at transposition distance $d - d_{\min}$ from $\rho$.
    Therefore, we may bound the growth of the ``entropy terms'' above by
    \begin{align*}
      &\hspace{-1cm} \#\{\rho \in S_k \text{ a } k\text{-cycle}: |\sigma\rho^{-1}| = d\} \\
      &\leq \#\{\rho \in S_k \text{ a } k\text{-cycle}: |\sigma\rho^{-1}| = d_{\min}\} \cdot \#\{\text{transpositions in } S_k\}^{d - d_{\min}} \\
      &\leq \#\{\rho \in S_k \text{ a } k\text{-cycle}: |\sigma\rho^{-1}| = d_{\min}\} \cdot (k^2)^{d - d_{\min}}.\numberthis
    \end{align*}
    So,
    \begin{align*}
      \sum_{\rho \in S_k \text{ a } k\text{-cycle}} x^{-|\sigma \rho^{-1}|}
      &\leq \#\{\rho \in S_k \text{ a } k\text{-cycle}: |\sigma\rho^{-1}| = d_{\min}\} \, x^{-d_{\min}} \sum_{d = d_{\min}}^{k - 1} \left(\frac{k^2}{x}\right)^{d - d_{\min}} \\
      &\leq \frac{1}{1 - \frac{k^2}{x}} \#\{\rho \in S_k \text{ a } k\text{-cycle}: |\sigma\rho^{-1}| = d_{\min}\} \, x^{-d_{\min}} \\
      &\leq 2\#\{\rho \in S_k \text{ a } k\text{-cycle}: |\sigma\rho^{-1}| = d_{\min}\} \, x^{-d_{\min}}. \numberthis
    \end{align*}

    It remains to bound the number of $k$-cycles at transposition distance $d_{\min}$ to $\sigma$.
    Such cycles are formed by repeatedly ``splicing'' one cycle of $\sigma$ into another.
    Suppose $\sigma$ has $m$ cycles.
    Then, the $k$-cycles at transposition dinstace $d_{\min}$ to $\sigma$ are enumerated (with some repetitions) by an ordering $C_1, \dots, C_m$ of the cycles of $\sigma$, a choice of one index in each of $[|C_2|], \dots, [|C_m|]$, and a choice of one index in each of $[|C_1|], [|C_1| + |C_2|], \dots, [|C_1| + \cdots + |C_{m - 1}|]$.
    Thus, we may bound
    \begin{align*}
      \#\{\rho \in S_k \text{ a } k\text{-cycle}: |\sigma\rho^{-1}| = d_{\min}\}
      &\leq m! \cdot \prod_{C \in \cyc(\sigma)} |C| \cdot \frac{(k - 1)!}{(k - m + 1)!} \\
      &\leq m^m \left(\frac{k}{m}\right)^m \frac{(k - 1)!}{(k - m + 1)!} \\
      &= \frac{(k - 1)!}{(k - m + 1)!} k^m, \numberthis
    \end{align*}
    as claimed.
\end{proof}

The following is the main technical result from which Theorem~\ref{thm:edge-km-unitary} will follows.
We isolate this result to point out that we do not require $\beta = \frac{1}{2}$ for this part of the argument; that restriction comes only from limitations of the proof of Theorem~\ref{thm:edge-km}, whose main assumption this Lemma verifies.
\begin{lemma}
    \label{lem:invariant-max-eval-bd}
    Suppose $k \geq 1 / (16 \min\{w(\alpha), w(\beta)\}^4)$ and $N \geq 2k^4 / \min\{w(\alpha), w(\beta)\}^2$.
    Let $\bU \in \sU(N)$ be a Haar-distributed unitary matrix.
    Let $\bA, \bD \in \{0, 1\}^{N \times N}$ be random diagonal matrices with $A_{ii} \sim \Ber(\alpha)$ and $D_{ii} \sim \Ber(\beta)$ independently of each other and of $\bU$.
    Let $\bB \colonequals \bU\bB\bU^*$.
    Then,
    \begin{equation}
        \left|\EE \Tr\left(\frac{\bA - \alpha \bm I_N}{\sqrt{\alpha(1 - \alpha)}}\frac{\bB - \beta \bm I_N}{\sqrt{\beta(1 - \beta)}}\right)^k\right| \lesssim \exp(4k^{3/4}).
    \end{equation}
\end{lemma}
\begin{proof}
    Let us write
    \begin{align}
      \what{\bA} &\colonequals \frac{1}{\sqrt{\alpha(1 - \alpha)}}(\bA - \alpha \bm I_N), \\
      \what{\bD} &\colonequals \frac{1}{\sqrt{\beta(1 - \beta)}}(\bD - \beta \bm I_N).
    \end{align}
    We begin by expanding the trace:
    \begin{align*}
      &\left|\EE \Tr(\what{\bA}\bU \what{\bD} \bU^*)^k\right| \\
      &= \left|\sum_{\substack{i_1, \dots, i_k \in [N] \\ j_1, \dots, j_k \in [N]}} \EE[\what{A}_{i_1i_1} \cdots \what{A}_{i_ki_k}] \,  \EE[\what{D}_{j_1j_1} \cdots \what{D}_{j_kj_k}] \, \EE[U_{i_1j_1} \overline{U_{i_2j_1}} U_{i_2 j_2} \overline{U_{i_3j_2}} \cdots U_{i_k j_k}\overline{U_{i_kj_1}}]\right|
      \intertext{In the expectation over $\bU$, we are in the situation treated by Theorem~\ref{thm:weingarten}, where $\bm i = (i_1, \dots, i_k)$, $\bm i^{\prime} = (i_2, \dots, i_k, i_1)$, and $\bm j = \bm j^{\prime} = (j_1,\dots, j_k)$. Let us write $\rho$ for the cycle permutation $(1 \, 2 \cdots \, n)$. Then, $\sigma \bm i = \bm i^{\prime} = \rho \bm i$ if and only if $\sigma = \rho \sigma^{\prime}$ for some $\sigma^{\prime}$ compatible with $\bm i$. Thus, we may apply Theorem~\ref{thm:weingarten} and expand}
      &= \left|\sum_{\substack{i_1, \dots, i_k \in [N] \\ j_1, \dots, j_k \in [N]}} \EE[\what{A}_{i_1i_1} \cdots \what{A}_{i_ki_k}] \,  \EE[\what{D}_{j_1j_1} \cdots \what{D}_{j_kj_k}] \, \sum_{\substack{\sigma \text{ compatible with } \bm i \\ \tau \text{ compatible with } \bm j}} W(\rho \sigma \tau^{-1})\right| \\
      &= \left|\sum_{\sigma, \tau \in S_k} W(\rho\sigma \tau^{-1})\left(\sum_{\substack{\bm i \in [N]^k \\ \text{compatible with } \sigma}} \EE[\what{A}_{i_1i_1} \cdots \what{A}_{i_ki_k}]\right)\left(\sum_{\substack{\bm j \in [N]^k \\ \text{ compatible with } \tau}} \EE[\what{D}_{j_1j_1} \cdots \what{D}_{j_kj_k}]\right)\right| \\
      &\leq \sum_{\sigma, \tau \in S_k} |W(\rho\sigma \tau^{-1})| \, f_{\alpha}(\sigma)f_{\beta}(\tau), \numberthis
    \end{align*}
    where we define
    \begin{equation}
        f_{\alpha}(\sigma) \colonequals \left|\sum_{\substack{\bm i \in [N]^k \\ \bm i \text{ compatible with } \sigma}} |\EE[\what{A}_{i_1i_1} \cdots \what{A}_{i_ki_k}]\right|,
    \end{equation}
    for $\what{A}_{11}, \dots, \what{A}_{NN}$ distributed i.i.d.\ as $\Ber(\alpha)$.
    We note that this is precisely the quantity controlled by Lemma~\ref{lem:f-sigma-bound}.

    First, by Theorem~\ref{thm:weingarten-bound}, we have
    \begin{equation}
        |W(\rho\sigma \tau^{-1})| \lesssim N^{-k}\left(\frac{4}{N}\right)^{|\rho \sigma \tau^{-1}|}.
    \end{equation}
    Also, by symmetry, the sum above is not changed by replacing $\rho$ with any other $k$-cycle.
    Thus, we may average over the $(k - 1)!$ many $k$-cycles in $S_k$, and, combining with the above bound, we have
    \begin{align*}
      &\hspace{-0.5cm}\EE \Tr(\what{\bA}\bU \what{\bD} \bU^*)^k \\
      &\lesssim \frac{1}{N^k(k - 1)!} \sum_{\rho \in S_k \text{ a } k\text{-cycle}} \sum_{\sigma, \tau \in S_k} \left(\frac{4}{N}\right)^{|\rho \sigma \tau^{-1}|} \, f_{\alpha}(\sigma)f_{\beta}(\tau) \\
      &= \frac{1}{N^k(k - 1)!}\sum_{\sigma, \tau \in S_k} \left(\sum_{\rho \in S_k \text{ a } k\text{-cycle}}\left(\frac{4}{N}\right)^{|\rho \sigma \tau^{-1}|}\right) \, f(\sigma)g(\tau)
        \intertext{where, by Lemma~\ref{lem:cycle-sum},}
      &\lesssim N^{-k} \sum_{\sigma, \tau \in S_k} \frac{k^{|\cyc(\sigma \tau^{-1})|}}{(k - |\cyc(\sigma\tau^{-1})| + 1)!} \left(\frac{4}{N}\right)^{-|\cyc(\sigma\tau^{-1})| + 1} f_{\alpha}(\sigma)f_{\beta}(\tau) \\
      &= N^{-k} \sum_{\sigma, \tau \in S_k} \frac{1}{(k - |\cyc(\sigma\tau^{-1})| + 1)! k^{|\cyc(\sigma \tau^{-1})| - 1}} \left(\frac{4k^2}{N}\right)^{-|\cyc(\sigma\tau^{-1})| + 1} f_{\alpha}(\sigma)f_{\beta}(\tau) \\
      &\leq \frac{1}{N^kk!} \sum_{\sigma, \tau \in S_k} \left(\frac{4k^2}{N}\right)^{-|\cyc(\sigma\tau^{-1})| + 1} f_{\alpha}(\sigma)f_{\beta}(\tau)
        \intertext{Now, we use that, since $\tau$ is a product of $\sum_{a = 1}^k (a - 1) |\cyc_a(\tau)|$ many transpositions, we have $|\cyc(\sigma\tau^{-1})| \geq |\cyc(\sigma)| - \sum_{a = 1}^k (a - 1) |\cyc_a(\tau)| = |\cyc(\sigma)| + |\cyc(\tau)| - k$. Substituting this bound in,}
      &\leq \frac{1}{(4k^2)^kk!} \left(\sum_{\sigma \in S_k} \left(\frac{4k^2}{N}\right)^{|\cyc(\sigma)|} f_{\alpha}(\sigma)\right)\left(\sum_{\tau \in S_k} \left(\frac{4k^2}{N}\right)^{|\cyc(\tau)|} f_{\beta}(\tau)\right). \numberthis \label{eq:invariant-model-initial}
    \end{align*}
    In this way, as a first step, we have eliminated the dependence of our expression on the special $k$-cycle $\rho$ and separated the dependence on $\alpha$ and $\beta$ into two different factors.

    It remains to bound each of the two sums appearing above.
    We work on the sum involving $f_{\alpha}(\sigma)$; the same argument will apply symmetrically to the other sum.
    By Lemma~\ref{lem:f-sigma-bound},
    \begin{align*}
      &\hspace{-0.5cm}\sum_{\sigma \in S_k} \left(\frac{4k^2}{N}\right)^{|\cyc(\sigma)|} f_{\alpha}(\sigma) \\
      &\lesssim \exp(k^{3/4})k \sum_{\sigma \in S_k} \left(\frac{4k^2}{N}\right)^{|\cyc(\sigma)|} w(\alpha)^{-\sum_{a = 3}^k (a - 2)|\cyc_a(\sigma)|} (|\cyc_1(\sigma)|)!! N^{|\cyc_1(\sigma)| / 2 + |\cyc_{\geq 2}(\sigma)|} \\
      &= \exp(k^{3/4})k \sum_{\sigma \in S_k} w(\alpha)^{-\sum_{a = 3}^k (a - 2)|\cyc_a(\sigma)|} (|\cyc_1(\sigma)|)!! N^{-|\cyc_1(\sigma)| / 2} (4k^2)^{|\cyc(\sigma)|}
        \intertext{and, since $(|\cyc_1(\sigma)|)!! \leq |\cyc_1(\sigma)|^{|\cyc_1(\sigma)| / 2} \leq k^{|\cyc_1(\sigma)| / 2}$,}
      &\leq \exp(k^{3/4})k \sum_{\sigma \in S_k} w(\alpha)^{-\sum_{a = 3}^k (a - 2)|\cyc_a(\sigma)|} \left(\frac{k}{N}\right)^{|\cyc_1(\sigma)| / 2} (4k^2)^{|\cyc(\sigma)|} \\
      &= \exp(k^{3/4})k \sum_{\sigma \in S_k} w(\alpha)^{-k + |\cyc_1(\sigma)| + 2 |\cyc_{\geq 2}(\sigma)|} \left(\frac{k}{N}\right)^{|\cyc_1(\sigma)| / 2} (4k^2)^{|\cyc(\sigma)|} \\
      &= \exp(k^{3/4})kw(\alpha)^{-k} \sum_{\sigma \in S_k} \left(\frac{k}{w(\alpha)^2N}\right)^{|\cyc_1(\sigma)| / 2} (4w(\alpha)^2k^2)^{|\cyc(\sigma)|} \\
      &= \exp(k^{3/4})kw(\alpha)^{-k} \sum_{a = 0}^k \binom{k}{a} \left(\frac{k}{w(\alpha)^2N}\right)^{a / 2} (4w(\alpha)^2k^2)^{a} \sum_{\substack{\sigma \in S_{k - a} \\ \cyc_1(\sigma) = \emptyset}} (4w(\alpha)^2k^2)^{|\cyc(\sigma)|}
      \intertext{Applying Proposition~\ref{prop:s1-sum} to the inner sum and bounding $k - a \leq k$,}
      &\leq \exp(2k^{3/4})k^2 k!! w(\alpha)^{-k} \sum_{a = 0}^k \binom{k}{a} \left(\frac{k}{w(\alpha)^2N}\right)^{a / 2} (4w(\alpha)^2k^2)^{(k + a) / 2} \\
      &= \exp(2k^{3/4})k^2 (4k^2)^{k/2} k!! \sum_{a = 0}^k \binom{k}{a} \left(\frac{4k^3}{N}\right)^{a / 2} \\
      &= \exp(2k^{3/4})k^2 \left(1 + \sqrt{\frac{4k^3}{N}}\right)^k (4k^2)^{k/2} k!! \\
      &\leq \exp\left(2k^{3/4} + \frac{2k^{5/2}}{N^{1/2}}\right) k^2 (4k^2)^{k/2} k!!. \numberthis
    \end{align*}

    Applying the same argument to the sum involving $f_{\beta}(\tau)$ in \eqref{eq:invariant-model-initial} and combining, we find
    \begin{align}
      \EE \Tr(\what{\bA}\bU \what{\bD} \bU^*)^k
      &\lesssim \exp\left(4k^{3/4} + \frac{4k^{5/2}}{N^{1/2}}\right) k^4 \frac{k!!^2}{k!}.
    \end{align}
    Finally, since by Stirling's approximation $k!!^2 / k! \lesssim \sqrt{k}$, the result follows.
\end{proof}

\begin{proof}[Proof of Theorem~\ref{thm:edge-km-unitary}]
    We let $\bB \colonequals \bU\bD\bU^{*}$, so that $\bA\bB\bA = \bA\bU\bD\bU^*\bA$ is the matrix under consideration.
    First, using Theorem~\ref{thm:weingarten-bound}, it is straightforward to show that Theorem~\ref{thm:untf} applies to this model, and thus the e.s.d.\ of $\bA\bB\bA$ converges in probability to $\MANOVA(\frac{1}{2}, \alpha)$.
    (Indeed, Theorem~\ref{thm:weingarten-bound} is overkill for this purpose, and much softer bounds on the joint moments of entries of $\bU$ would suffice.)
    Taking $k = cN^{1/4}$ for a sufficiently small constant $c > 0$, we see that the assumptions of Lemma~\ref{lem:invariant-max-eval-bd} are satisfied for sufficiently large $\iota$.
    Thus, all conditions of Theorem~\ref{thm:edge-km} are satisfied, and the result follows.
\end{proof}

\section*{Acknowledgments}
\addcontentsline{toc}{section}{Acknowledgments}

Thanks to Dustin Mixon for first introducing me to the conjectures of \cite{HZG-2017-MANOVA} and for several stimulating conversations about them, and to March Boedihardjo for showing me the reference \cite{AF-2014-AsymptoticallyLiberatingUnitaryMatrices} and discussing asymptotic freeness of idempotents.
Thanks also to Daniel Spielman and Xifan Yu for helpful discussions at the beginning of this project.

\clearpage

\addcontentsline{toc}{section}{References}
\bibliographystyle{alpha}
\bibliography{main}

\newcommand{\etalchar}[1]{$^{#1}$}
\begin{thebibliography}{HGMZ21}

\bibitem[AEK06]{AEK-2006-PrincipalAngleRandomSubspaces}
Pierre-Antoine Absil, Alan Edelman, and Plamen Koev.
\newblock On the largest principal angle between random subspaces.
\newblock {\em Linear Algebra and its Applications}, 414(1):288--294, 2006.

\bibitem[AF14]{AF-2014-AsymptoticallyLiberatingUnitaryMatrices}
Greg~W Anderson and Brendan Farrell.
\newblock Asymptotically liberating sequences of random unitary matrices.
\newblock {\em Advances in Mathematics}, 255:381--413, 2014.

\bibitem[AGZ10]{AGZ-2010-RandomMatrices}
Greg~W Anderson, Alice Guionnet, and Ofer Zeitouni.
\newblock {\em An introduction to random matrices}.
\newblock Cambridge University Press, 2010.

\bibitem[AS64]{AS-1964-HandbookMathematicalFunctions}
Milton Abramowitz and Irene~A Stegun.
\newblock {\em Handbook of mathematical functions with formulas, graphs, and
  mathematical tables}, volume~55.
\newblock US Government Printing Office, 1964.

\bibitem[CC04]{CC-2004-AsymptoticFreenessGaussianWishart}
Mireille Capitaine and Muriel Casalis.
\newblock Asymptotic freeness by generalized moments for {Gaussian} and
  {Wishart} matrices. {Application} to beta random matrices.
\newblock {\em Indiana University Mathematics Journal}, pages 397--431, 2004.

\bibitem[CK14]{CK-2014-LiberationProjections}
Beno{\^\i}t Collins and Todd Kemp.
\newblock Liberation of projections.
\newblock {\em Journal of Functional Analysis}, 266(4):1988--2052, 2014.

\bibitem[CM17]{CM-2017-WeingartenCalculusBounds}
Beno{\^\i}t Collins and Sho Matsumoto.
\newblock Weingarten calculus via orthogonality relations: new applications.
\newblock {\em Latin American Journal of Probability and Mathematical
  Statistics (ALEA)}, 14(1):631--656, 2017.

\bibitem[Col05]{Collins-2005-ProductRandomProjections}
Beno{\^\i}t Collins.
\newblock Product of random projections, {Jacobi} ensembles and universality
  problems arising from free probability.
\newblock {\em Probability Theory and Related Fields}, 133(3):315--344, 2005.

\bibitem[C{\'S}06]{CS-2006-HaarMeasureMoments}
Beno{\^\i}t Collins and Piotr {\'S}niady.
\newblock Integration with respect to the {Haar} measure on unitary, orthogonal
  and symplectic group.
\newblock {\em Communications in Mathematical Physics}, 264(3):773--795, 2006.

\bibitem[DE15]{DE-2015-InfiniteRandomMatrixTheory}
Alexander Dubbs and Alan Edelman.
\newblock Infinite random matrix theory, tridiagonal bordered {Toeplitz}
  matrices, and the moment problem.
\newblock {\em Linear Algebra and its Applications}, 467:188--201, 2015.

\bibitem[Dem08]{Demni-2008-FreeJacobiProcess}
Nizar Demni.
\newblock Free {Jacobi} process.
\newblock {\em Journal of Theoretical Probability}, 21(1):118--143, 2008.

\bibitem[DHH12]{DHH-2012-FreeJacobiProcessSpectral}
Nizar Demni, Tarek Hamdi, and Taoufik Hmidi.
\newblock Spectral distribution of the free {Jacobi} process.
\newblock {\em Indiana University Mathematics Journal}, pages 1351--1368, 2012.

\bibitem[DK08]{DK-2008-ExtremeEigenvaluesJacobi}
Ioana Dumitriu and Plamen Koev.
\newblock Distributions of the extreme eigenvalues of beta-{Jacobi} random
  matrices.
\newblock {\em SIAM Journal on Matrix Analysis and Applications}, 30(1):1--6,
  2008.

\bibitem[DZ09]{DZ-2009-JacobiProcessLargeDeviations}
Nizar Demni and Marguerite Zani.
\newblock Large deviations for statistics of the {Jacobi} process.
\newblock {\em Stochastic Processes and their Applications}, 119(2):518--533,
  2009.

\bibitem[EF13]{EF-2013-MANOVALocalDensity}
L{\'a}szl{\'o} Erd{\H{o}}s and Brendan Farrell.
\newblock Local eigenvalue density for general {MANOVA} matrices.
\newblock {\em Journal of Statistical Physics}, 152(6):1003--1032, 2013.

\bibitem[ENV07]{ENV-1998-SaintFlourNotes}
Michel Emery, Arkadi Nemirovski, and Dan Voiculescu.
\newblock {\em Lectures on Probability Theory and Statistics: Ecole D'Ete de
  Probabilites de Saint-Flour XXVIII-1998}.
\newblock Springer, 2007.

\bibitem[ES08]{ES-2008-BetaJacobiMatrixModel}
Alan Edelman and Brian~D Sutton.
\newblock The beta-{Jacobi} matrix model, the {CS} decomposition, and
  generalized singular value problems.
\newblock {\em Foundations of Computational Mathematics}, 8(2):259--285, 2008.

\bibitem[Far11]{Farrell-2011-LimitingEmpirical}
Brendan Farrell.
\newblock Limiting empirical singular value distribution of restrictions of
  discrete {Fourier} transform matrices.
\newblock {\em Journal of Fourier Analysis and Applications}, 17(4):733--753,
  2011.

\bibitem[FK81]{FK-1981-EigenvaluesRandomMatrices}
Zolt{\'a}n F{\"u}redi and J{\'a}nos Koml{\'o}s.
\newblock The eigenvalues of random symmetric matrices.
\newblock {\em Combinatorica}, 1(3):233--241, 1981.

\bibitem[FN15]{FN-2013-TruncationsTensorHaarUnitary}
Brendan Farrell and Raj~Rao Nadakuditi.
\newblock Local spectrum of truncations of {Kronecker} products of {Haar}
  distributed unitary matrices.
\newblock {\em Random Matrices: Theory and Applications}, 4(1), 2015.

\bibitem[Gem80]{Geman-1980-LimitNormRandomMatrix}
Stuart Geman.
\newblock A limit theorem for the norm of random matrices.
\newblock {\em The Annals of Probability}, pages 252--261, 1980.

\bibitem[HGMZ21]{HGMZ-2021-AsymptoticFrameTheory}
Marina Haikin, Matan Gavish, Dustin~G Mixon, and Ram Zamir.
\newblock Asymptotic frame theory for analog coding.
\newblock {\em Foundations and Trends{\textregistered} in Communications and
  Information Theory}, 18(4):526--645, 2021.

\bibitem[How80]{Howard-1980-AssociatedStirlingNumbers}
FT~Howard.
\newblock Associated {Stirling} numbers.
\newblock {\em Fibonacci Quarterly}, 18(4):303--315, 1980.

\bibitem[HZ21]{HZ-2021-MomentsSubsetsETF}
Marina Haikin and Ram Zamir.
\newblock Moments of subsets of general equiangular tight frames.
\newblock {\em arXiv preprint arXiv:2107.00888}, 2021.

\bibitem[HZG17]{HZG-2017-MANOVA}
Marina Haikin, Ram Zamir, and Matan Gavish.
\newblock Random subsets of structured deterministic frames have {MANOVA}
  spectra.
\newblock {\em Proceedings of the National Academy of Sciences},
  114(26):E5024--E5033, 2017.

\bibitem[Jia13]{Jiang-2013-LimitTheoremsJacobiEnsembles}
Tiefeng Jiang.
\newblock Limit theorems for beta-{Jacobi} ensembles.
\newblock {\em Bernoulli}, 19(3):1028--1046, 2013.

\bibitem[Joh08]{Johnstone-2008-JacobiMatrixLargestEigenvalue}
Iain~M Johnstone.
\newblock Multivariate analysis and {Jacobi} ensembles: Largest eigenvalue,
  {Tracy}--{Widom} limits and rates of convergence.
\newblock {\em Annals of Statistics}, 36(6):2638, 2008.

\bibitem[LM23]{LM-2020-FreeProbabilityKestenMcKay}
Iris Stephanie~Arenas Longoria and James~A Mingo.
\newblock Freely independent coin tosses, standard {Young} tableaux, and the
  {Kesten}--{McKay} law.
\newblock {\em American Mathematical Monthly}, 130(1):35--48, 2023.

\bibitem[MMP21]{MMP-2019-RandomSubensembles}
Mark Magsino, Dustin~G Mixon, and Hans Parshall.
\newblock {Kesten}--{McKay} law for random subensembles of {Paley} equiangular
  tight frames.
\newblock {\em Constructive Approximation}, 53:381--402, 2021.

\bibitem[MS17]{MS-2017-FreeProbabilityRandomMatrices}
James~A Mingo and Roland Speicher.
\newblock {\em Free probability and random matrices}, volume~35.
\newblock Springer, 2017.

\bibitem[Nov14]{Novak-2014-ThreeLectures}
Jonathan Novak.
\newblock Three lectures on free probability.
\newblock {\em Random matrix theory, interacting particle systems, and
  integrable systems}, 65(309-383):13, 2014.

\bibitem[NS06]{NS-2006-LecturesCombinatoricsFreeProbability}
Alexandru Nica and Roland Speicher.
\newblock {\em Lectures on the combinatorics of free probability}, volume~13.
\newblock Cambridge University Press, 2006.

\bibitem[RK16]{RK-2016-MANOVAColumnSubsampled}
Raviv Raich and Jinsub Kim.
\newblock On the eigenvalue distribution of column sub-sampled semi-unitary
  matrices.
\newblock In {\em 2016 IEEE Statistical Signal Processing Workshop (SSP)},
  pages 1--5. IEEE, 2016.

\bibitem[SSB{\etalchar{+}}91]{SSBCHMW-1991-RiordanGroup}
Louis Shapiro, Renzo Sprugnoli, Paul Barry, Gi-Sang Cheon, Tian-Xiao He,
  Donatella Merlini, and Weiping Wang.
\newblock The {Riordan} group.
\newblock In {\em The Riordan Group and Applications}, pages 47--67. Springer,
  1991.

\bibitem[Tao12]{Tao-2012-RandomMatrixTheory}
Terence Tao.
\newblock {\em Topics in random matrix theory}, volume 132.
\newblock American Mathematical Soc., 2012.

\bibitem[TCSV10]{TCSV-2010-CapacityChannelsFadingWachter}
Antonia~M Tulino, Giuseppe Caire, Shlomo Shamai, and Sergio Verd{\'u}.
\newblock Capacity of channels with frequency-selective and time-selective
  fading.
\newblock {\em IEEE Transactions on Information Theory}, 56(3):1187--1215,
  2010.

\bibitem[TV04]{TV-2004-RandomMatrixWirelessCommunications}
Antonia~M Tulino and Sergio Verd{\'u}.
\newblock Random matrix theory and wireless communications.
\newblock {\em Foundations and Trends{\textregistered} in Communications and
  Information Theory}, 1(1):1--182, 2004.

\bibitem[TVCS07]{TVCS-2007-GaussianErasureChannel}
Antonia Tulino, Sergio Verd{\'u}, Giuseppe Caire, and Shlomo Shamai.
\newblock The {Gaussian} erasure channel.
\newblock In {\em 2007 IEEE International Symposium on Information Theory},
  pages 1721--1725. IEEE, 2007.

\bibitem[VDN92]{VDN-1992-FreeRandomVariables}
Dan~V Voiculescu, Ken~J Dykema, and Alexandru Nica.
\newblock {\em Free random variables}.
\newblock Number~1. American Mathematical Soc., 1992.

\bibitem[Voi91]{Voiculescu-1991-RandomMatricesFreeProducts}
Dan Voiculescu.
\newblock Limit laws for random matrices and free products.
\newblock {\em Inventiones Mathematicae}, 104(1):201--220, 1991.

\bibitem[Wac80]{Wachter-1980-EmpiricalMeasureDiscriminantRatios}
Kenneth~W Wachter.
\newblock The limiting empirical measure of multiple discriminant ratios.
\newblock {\em The Annals of Statistics}, pages 937--957, 1980.

\bibitem[Wal18]{Waldron-2018-FiniteTightFrames}
Shayne~FD Waldron.
\newblock {\em An introduction to finite tight frames}.
\newblock Springer, 2018.

\end{thebibliography}

\clearpage

\appendix

\section{MANOVA distribution parametrizations}
\label{app:manova-param}

In this appendix, we detail how our parametrization of the MANOVA family of distributions relates to others that have appeared in the literature.
In particular, we clarify its relationship with the parametrization of \cite{DE-2015-InfiniteRandomMatrixTheory}, whose results on orthogonal polynomials we apply in Appendix~\ref{app:gf}, and with that of \cite{HGMZ-2021-AsymptoticFrameTheory}, who state their conjecture that we prove in the main text in terms of yet another parametrization.

\subsection{Parametrization of \cite{DE-2015-InfiniteRandomMatrixTheory}}
\label{sec:param-de}

The MANOVA distribution is parametrized in \cite{DE-2015-InfiniteRandomMatrixTheory} by parameters $a, b > 0$, which are related to our $\alpha, \beta \in (0, 1)$ by
\begin{align}
  a &= \frac{\beta}{\alpha}, \\
  b &= \frac{1}{\alpha} - a = \frac{1 - \beta}{\alpha},
      \intertext{and we observe that we also have}
      a + b &= \frac{1}{\alpha}.
\end{align}
In \cite{DE-2015-InfiniteRandomMatrixTheory} they work under the assumption that $a, b \geq 1$, which for our parameters translates to $\alpha \leq \min\{\beta, 1 - \beta\}$.
In particular, under this assumption we have that the mass of the atom at 1 of our distribution is
\begin{equation}
    \max\{\alpha + \beta - 1, 0\} = \max\{\alpha - (1 - \beta), 0\} = 0,
\end{equation}
so the results of \cite{DE-2015-InfiniteRandomMatrixTheory} pertain only to the case where there is no atom at 1.

On the other hand, the mass of the atom at 0 is
\begin{equation}
    1 - \min\{\alpha, \beta\} = 1 - \alpha > 0,
\end{equation}
since $\alpha \leq \frac{1}{2}$ under the above assumption.
However, the distribution \cite{DE-2015-InfiniteRandomMatrixTheory} work with omits this atom and renormalizes the distribution.
The reason for this is that these authors are interested in the \emph{Jacobi ensemble} of random matrices, whose eigenvalues are related to a product of random projections but are never zero; see the definition in our Section~\ref{sec:related}.

As a sanity check, we may verify that the endpoints of the support of the absolutely continuous part of their distribution are the same as ours; following their Table 1,
\begin{align*}
  r_{\pm}
  &= \left(\frac{\sqrt{b} \pm \sqrt{a(a + b - 1)}}{a + b}\right)^2 \\
  &= \left(\frac{\sqrt{\frac{1 - \beta}{\alpha}} \pm \sqrt{\frac{\beta}{\alpha}(\frac{1}{\alpha} - 1)}}{\frac{1}{\alpha}}\right)^2 \\
  &= \left(\sqrt{\alpha(1 - \beta)} \pm \sqrt{\beta(1 - \alpha)}\right)^2. \numberthis
\end{align*}
And, again from their Table 1, the density they work with is indeed identical to ours except for being larger by an extra factor of $a + b = \frac{1}{\alpha}$, which is precisely the renormalization without the zero atom we have claimed.

To help the reader parse our discussion in Appendix~\ref{app:gf}, we also give a few further equations between the $\alpha, \beta$ parameters and the $a, b$ parameters that appear in the orthogonal polynomial recursion of Proposition~\ref{prop:de-orth-poly}:
\begin{align}
  \frac{a}{a + b} &= \frac{\frac{\beta}{\alpha}}{\frac{1}{\alpha}} = \beta, \\
  \frac{a + b}{a + b - 1} &= \frac{\frac{1}{\alpha}}{\frac{1}{\alpha} - 1} = \frac{1}{1 - \alpha}, \\
  \sqrt{ab(a + b - 1)} &= \sqrt{\frac{\beta(1 - \beta)}{\alpha^2}\left(\frac{1}{\alpha} - 1\right)} = \frac{1}{\alpha^{3/2}} \sqrt{(1 - \alpha)\beta(1 - \beta)}, \\
  \frac{\sqrt{ab(a + b - 1)}}{(a + b)^2} &= \sqrt{\alpha(1 - \alpha)\beta(1 - \beta)}, \\
  b + a(a + b - 1) &= \frac{1 - \beta}{\alpha} + \frac{\beta}{\alpha}\left(\frac{1}{\alpha} - 1\right) = \frac{\alpha(1 - \beta) + \beta(1 - \alpha)}{\alpha^2}, \\
  \frac{\sqrt{ab(a + b - 1)}}{b + a(a + b - 1)} &= \frac{\sqrt{\alpha(1 - \alpha)\beta(1 - \beta)}}{\alpha(1 - \beta) + \beta(1 - \alpha)}.
\end{align}

\subsection{Parametrization of \cite{HZG-2017-MANOVA, HGMZ-2021-AsymptoticFrameTheory}}

The MANOVA distribution is parametrized in \cite{HZG-2017-MANOVA, HGMZ-2021-AsymptoticFrameTheory} (which state the conjecture that Theorem~\ref{thm:untf} proves) by parameters $\widetilde{\beta} > 0$ and $\widetilde{\gamma} \in (0, 1)$.
(Their work simply calls them $\beta, \gamma$, but we rename to $\widetilde{\beta}, \widetilde{\gamma}$ to avoid confusion between our and their $\beta$ parameters.)
These are related to our $\alpha, \beta$ by
\begin{align}
  \widetilde{\beta} &= \frac{\alpha}{\beta}, \\
  \widetilde{\gamma} &= \beta.
\end{align}
Their MANOVA distribution is also rescaled by $\frac{1}{\widetilde{\gamma}} = \frac{1}{\beta}$.
After accounting for this rescaling, we may again verify that the endpoints of the support of the absolutely continuous part of their distribution are the same as ours; following their Equation (9), and writing $\widetilde{r}_{\pm}$ for the endpoints of their distribution's support,
\begin{align*}
    \widetilde{r}_{\pm}
    &= \left( \sqrt{\widetilde{\beta}(1 - \widetilde{\gamma})} \pm \sqrt{1 - \widetilde{\beta}\widetilde{\gamma}}\right)^2 \\
    &= \left( \sqrt{\frac{\alpha}{\beta}(1 - \beta)} \pm \sqrt{1 - \alpha}\right)^2 \\
    &= \frac{1}{\beta} \left( \sqrt{\alpha(1 - \beta)} \pm \sqrt{\beta(1 - \alpha)}\right)^2 \\
    &= \frac{1}{\widetilde{\gamma}}\, r_{\pm}. \numberthis
\end{align*}
We note that this parametrization, unlike that of \cite{DE-2015-InfiniteRandomMatrixTheory}, can include atoms both at~0 and at~1.

\section{MANOVA moment recursion revisited: Direct proof of Lemma~\ref{lem:manova-recursion}}
\label{app:gf}

We give an alternative approach to proving a special case of the recursion of Lemma~\ref{lem:manova-recursion} for the moments of $\MANOVA(\alpha, \beta)$, which we believe might be interesting in its own right (we developed this proof before finding the free probability ``shortcuts'' used in the main text) and which will develop some generating function tools used in the proof of Theorem~\ref{thm:edge-km}.

Let us sketch the general proof strategy first.
We will only consider the case
\begin{equation}
    \alpha \leq \min\{\beta, 1 - \beta\}.
\end{equation}
First, it will turn out to be more useful to use the following adjusted version of the MANOVA family.
\begin{definition}[\cite{DE-2015-InfiniteRandomMatrixTheory}'s MANOVA distributions]
    Suppose $\alpha \leq \min\{\beta, 1 - \beta\}$.
    Write $\DE(\alpha, \beta)$ for the law of $X \sim \MANOVA(\alpha, \beta)$ conditional on $X > 0$.
\end{definition}
\noindent
The moments of this distribution have the following simple relation to the $m_k$ above.
Write
\begin{equation}
    m_k^{\prime} \colonequals \EE_{X \sim \DE(\alpha, \beta)}[X^k].
\end{equation}
Then, for all $k \geq 1$, since for $X \sim \MANOVA(\alpha, \beta)$ we have $\PP[X = 0] = 1 - \alpha$, we have
\begin{equation}
    m_k = \Ex_{X \sim \MANOVA(\alpha, \beta)}X^k = \Px_{X \sim \MANOVA(\alpha, \beta)}[X > 0]\Ex_{X \sim \MANOVA(\alpha, \beta)}[X^k \mid X > 0] = \alpha m_k^{\prime}.
\end{equation}
Then, our claim in Lemma~\ref{lem:manova-recursion} is equivalent to showing that, for all $k \geq 0$,
\begin{equation}
    \widetilde{m}_k = \alpha \Ex_{X \sim \DE(\alpha, \beta)} \left(\frac{X}{\beta} - 1\right)^k
\end{equation}

This says that the moments (after the affine transformation $x \mapsto \frac{x}{\beta} - 1$) of $\DE(\alpha, \beta)$ satisfy a certain recursion.
Let us express this claim as one about expectations of polynomials associated to this recursion.
\begin{definition}
    For each $k \geq 1$, define the polynomial
    \begin{equation}
        s_k(x) \colonequals \sum_{a = 0}^{k} \left(\frac{x}{\beta} - 1\right)^{a} \sum_{j = 1}^{k} (-\alpha)^{k - j} q_{k, j, a}(\beta^{-1} - 1) - \alpha^{k - 1}\left(\sum_{b = 1}^{k - 1}(1 - \beta^{-1})^b\right).
    \end{equation}
\end{definition}
\noindent
Then, Lemma~\ref{lem:manova-recursion} will follow from showing that
\begin{equation}
    \Ex_{X \sim \DE(\alpha, \beta)} s_k(X) = 0,
\end{equation}
where we observe that $q_{k, k, k}(x) = 1$ for all $k \geq 1$, so we equivalently have
\begin{align*}
  s_k(x) &= \left(\frac{x}{\beta} - 1\right)^k + \sum_{a = 0}^{k - 1} \left(\frac{x}{\beta} - 1\right)^{a} \sum_{j = 1}^{k} (-\alpha)^{k - j} q_{k, j, a}(\beta^{-1} - 1) \\
  &\hspace{2.42cm} - \alpha^{k - 1}\left(\sum_{b = 1}^{k - 1}(1 - \beta^{-1})^b\right), \numberthis
\end{align*}
which corresponds more transparently to the recursion in Lemma~\ref{lem:manova-recursion}.
(Here, $q_{k, k, a}(z) = \One\{k = a\}$, so changing the limit of the $j$ summation from $k$ to $k - 1$ does not introduce a discrepancy between this and the recursion in Lemma~\ref{lem:manova-recursion}.)

We will actually do more than that; we will instead derive (in an implicit form in terms of generating functions) the expansion of $s_k(x)$ in orthogonal polynomials for $\DE(\alpha, \beta)$, from which we will be able to read off that the coefficient of the constant polynomial in this expansion is zero.

We execute our plan in two steps: first, we compute a generating function for the $s_k(x)$, which mostly amounts to computing a generating function for the $q_{k, j, a}(x)$.
Second, we compute a generating function for a sequence of orthogonal polynomials for $\DE(\alpha, \beta)$, which fortunately follows directly from prior work of \cite{DE-2015-InfiniteRandomMatrixTheory}.
With these two expressions in hand, the result will amount to an algebraic manipulation.

We first turn to computing a generating function for the $s_k(x)$.
The following combinatorial definition will turn out to be crucial.
\begin{definition}[Cyclic runs]
    For a subset $T \subseteq [n]$ with $0 < |T| < n$, view $T$ as a subset of the vertices of the $n$-cycle.
    Then, let $\runs(T)$ denote the number of connected components in the subgraph of the $n$-cycle induced by $T$.
    We further extend this definition to $\runs(\emptyset) = \runs([n]) = 0$.
\end{definition}
\noindent
It may seem unusual to set $\runs([n]) = 0$ rather than $\runs([n]) = 1$, but we will see that this is natural in our setting: the key point will be that we need to count the ``breaks'' between runs in $T$ rather than runs of elements of $T$ themselves.

\begin{proposition}[Cyclic runs generating function]
    \label{prop:sets-runs-gf}
    As a formal power series,
    \begin{equation}
        R(x, y, z) \colonequals \sum_{n \geq 1} \sum_{T \subseteq [n]} x^{n} y^{|T|} z^{\runs(T)} = \frac{2x^2yz + xy(1 - x) + x(1 - xy)}{(1 - xy)(1 - x) - x^2yz}.
    \end{equation}
\end{proposition}
\begin{proof}
    We view this generating function as a sum of four generating functions, depending on the membership of $1$ and $n$ in $T$.
    These may all be expressed in terms of the generating function
    \begin{equation}
        F(x, y, z) = \sum_{k \geq 0}\left(\left(z\sum_{i \geq 1}x^iy^i\right)\left(\sum_{i \geq 1}x^i\right)\right)^k = \frac{1}{1 - \frac{x^2yz}{(1 - xy)(1 - x)}}.
    \end{equation}
    The four summands are:
    \begin{align}
      1 \in T, n \notin T &\rightsquigarrow F(x, y, z) \left(z\sum_{i \geq 1}x^iy^i\right)\left(\sum_{i \geq 1}x^i\right) = F(x, y, z) \frac{x^2yz}{(1 - xy)(1 - x)} , \\
      1 \notin T, n \in T &\rightsquigarrow F(x, y, z) \left(z\sum_{i \geq 1}x^iy^i\right)\left(\sum_{i \geq 1}x^i\right) = F(x, y, z) \frac{x^2yz}{(1 - xy)(1 - x)}, \\
      1 \in T, n \in T, &\rightsquigarrow F(x, y, z)\sum_{i \geq 1}x^iy^i = F(x, y, z) \frac{xy}{1 - xy}, \\
      1 \notin T, n \notin, T &\rightsquigarrow F(x, y, z) \sum_{i \geq 1} x^i = F(x, y, z) \frac{x}{1 - x}.
    \end{align}
    The final generating function is then the sum of these upon simplifying.
\end{proof}

\begin{proposition}[$q_{k, j, a}$ generating function]
    \label{prop:q-gf}
    As a formal power series,
    \begin{align*}
      Q(w, x, y, z)
      &\colonequals \sum_{k \geq 1}\sum_{j = 0}^k \sum_{a = 0}^j q_{k, j, a}(w) x^k y^j z^a \\
      &= \frac{xy(-w^2x^2 + wx^2z + wx^2 + 2wx + z)}{(wx^2y + wx^2 + xyz + wx - x - 1)(wx - 1)(x + 1)}. \numberthis
    \end{align*}
\end{proposition}
\begin{proof}
    We begin working with the generating function up to sign flips in $w$, $x$, and $y$:
    \begin{align*}
      Q(-w, -x, -y, z)
      &= \sum_{k \geq 1}\sum_{j = 1}^k \sum_{a = 0}^j q_{k, j, a}(-w) (-x)^k (-y)^j z^a \\
      &= \sum_{k \geq 1} x^k \sum_{j = 1}^k y^j \sum_{S \in \binom{[k]}{j}} \prod_{i = 1}^j\left(zc_{p_i(S)}(-w) + d_{p_i(S)}(-w)\right) \\
      &= \sum_{k \geq 1} x^k \sum_{j = 1}^k y^j \sum_{S \in \binom{[k]}{j}} \prod_{i = 1}^j\left(z + (z + 1)\sum_{b = 1}^{p_i(S) - 1} w^b\right). \numberthis
    \end{align*}
    Consider expanding the innermost product.
    We introduce the following bookkeeping mechanism for the terms that arise: consider the $k$-cycle, oriented clockwise, and labelled on both vertices and edges such that edge $i$ points from vertex $i$ to vertex $i + 1$ (modulo $k$).
    View each $S \subseteq [k]$ as a set of edges of this $k$-cycle.
    To each $S = \{s_1 < \cdots < s_j\}$ and a choice of numbers $0 \leq r_i \leq p_i(S) - 1$ for each $i \in [|S|]$, associate the set $T \subseteq [k]$ of vertices that, for each $r_i > 0$, contains vertex $s_i + 1$ together with the subsequent $r_i - 1$ vertices (i.e., the sequence of $r_i - 1$ consecutive vertices starting from $s_i + 1$ and following the orientation of the cycle).
    See Figure~\ref{fig:cycle-S-T} for an illustration of this construction.
    This $T$ has $|T| = r_1 + \dots + r_j$, and has $\runs(T) = \#\{r_i: r_i > 0\}$.

    Call $S$ and $T$ \emph{compatible} if it is possible to obtain $T$ from $S$ for some choice of $r_i$.
    This is the case if and only if the following two conditions hold: $S$ contains the edge pointing to the first vertex of every run of $T$, and $S$ does not contain any edge pointing from a vertex of $T$.
    (Note that, by the definition of a run, no edge pointing to the first vertex of a run can point from a vertex of $T$.)
    To count the $S \in \binom{[k]}{j}$ that are compatible with a given $T$, we observe that the above conditions specify $\runs(T)$ many elements of $S$, and prohibit $|T|$ many elements of $S$; thus, the number of such $S$ is $\binom{k - |T| - \runs(T)}{j - \runs(T)}$.

    Then, we may reorganize the sums in the above calculation as
    \begin{align*}
      &\hspace{-0.0cm}Q(-w, -x, -y, z) \\
      &= \sum_{k \geq 1} x^k \sum_{j = 1}^k y^j \sum_{S \in \binom{[k]}{j}} \prod_{i = 1}^j\left(z + (z + 1)\sum_{b = 1}^{p_i(S) - 1} w^b\right)
        \intertext{where, expanding the product using the above definitions,}
      &= \sum_{k \geq 1} x^k\sum_{j = 1}^k y^j \sum_{T \subseteq [k]} \#\left\{S \in \binom{[k]}{j}: S \text{ compatible with } T \right\} w^{|T|} (z + 1)^{\runs(T)} z^{j - \runs(T)} \\
      &= \sum_{k \geq 1} x^k\sum_{T \subseteq [k]} w^{|T|} \left(\frac{z + 1}{z}\right)^{\runs(T)} \sum_{j = 1}^k \binom{k - |T| - \runs(T)}{j - \runs(T)} (yz)^j
        \intertext{and, reindexing}
      &= \sum_{k \geq 1} x^k\sum_{T \subseteq [k]} w^{|T|} \left(\frac{z + 1}{z}\right)^{\runs(T)} \\
      &\hspace{1.0cm}\left(\sum_{j = 0}^{k - |T| - \runs(T)} \binom{k - |T| - \runs(T)}{j} (yz)^{j + \runs(T)} - \One\{T = \emptyset\} - \One\{T = [k]\}\right)
        \intertext{and, by the binomial theorem, }
      &= \sum_{k \geq 1} x^k\sum_{T \subseteq [k]} w^{|T|} \left(\frac{z + 1}{z}\right)^{\runs(T)} (yz)^{\runs(T)}(1 + yz)^{k - |T| - \runs(T)} - \frac{x}{1 - x} - \frac{wx}{1 - wx} \\
      &= \sum_{k \geq 1} \sum_{T \subseteq [k]} (x + xyz)^k \left(\frac{w}{1 + yz}\right)^{|T|} \left(\frac{y + yz}{1 + yz}\right)^{\runs(T)} - \frac{x}{1 - x} - \frac{wx}{1 - wx} \\
      &= R\left(x + xyz, \frac{w}{1 + yz}, \frac{y + yz}{1 + yz}\right) - \frac{x}{1 - x} - \frac{wx}{1 - wx}, \numberthis
    \end{align*}
    and finally an application of Proposition~\ref{prop:sets-runs-gf} and some algebraic simplification gives the result.
\end{proof}

\begin{figure}
    \label{fig:cycle-S-T}
    \begin{center}
        \includegraphics[scale=0.5]{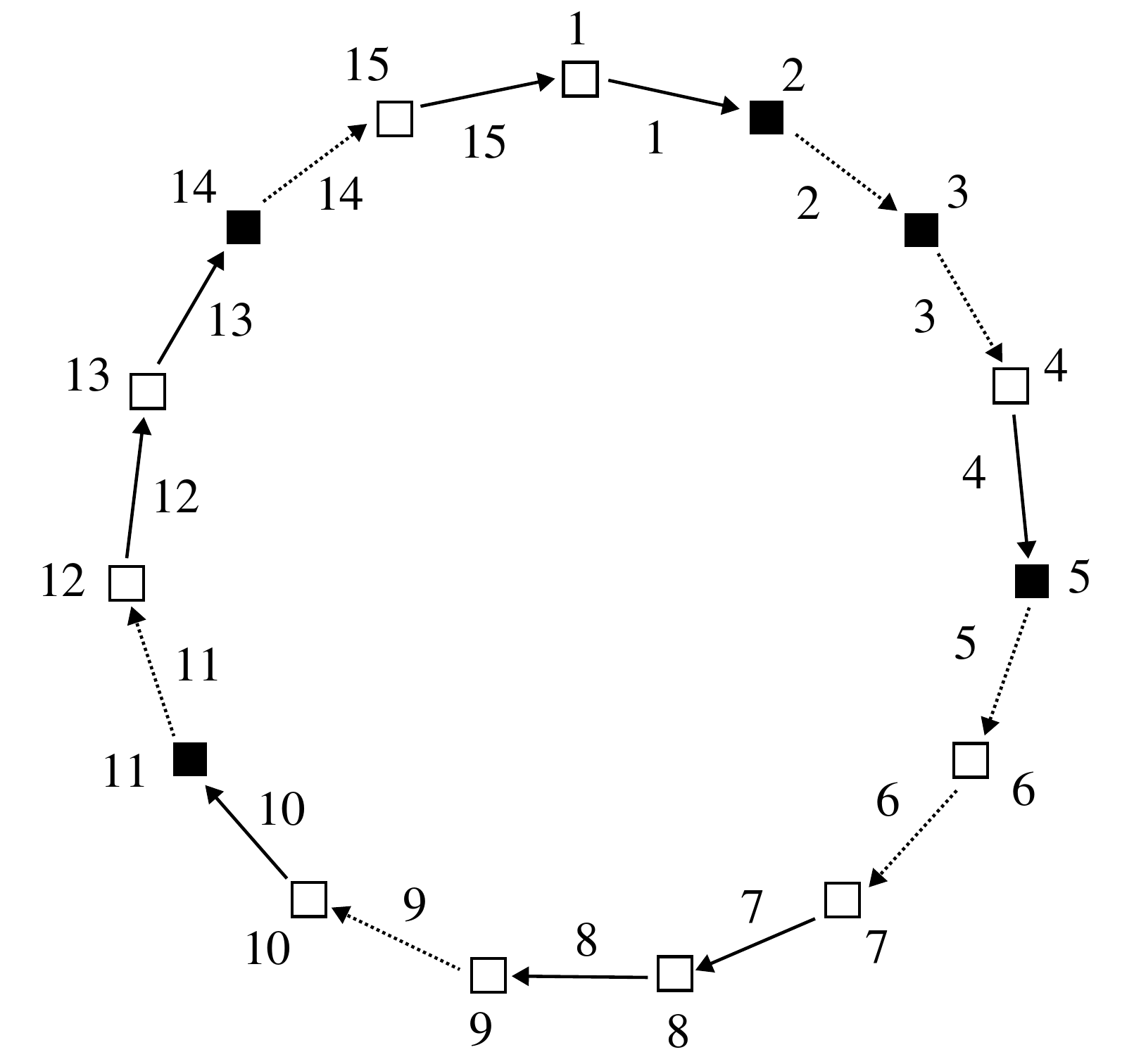}
    \end{center}
    \vspace{-1em}
    \caption{A valid configuration of $S = \{1,4, 7, 8, 10, 12,13, 15\}$ and $T = \{2, 3, 5, 11, 14\}$ subsets of edge and vertex sets, respectively, of an oriented cycle, as in the proof of Proposition~\ref{prop:q-gf}. Solid arrows represent edges in $S$ and dotted arrows edges not in $S$, and filled boxes represent vertices in $T$ and empty boxes vertices not in $T$.
      Vertices are labelled outside the cycle and edges inside the cycle.}
\end{figure}

Finally, we are ready to compute the generating function of the $s_k(x)$ and complete the first phase of our plan.
\begin{proposition}[$s_k$ generating function]
    As a formal power series,
    \begin{align}
      S(x, t) \colonequals \sum_{k \geq 1} s_k(x) t^k = Q\left(\frac{1}{\beta} - 1, -\alpha t, -\frac{1}{\alpha}, \frac{x}{\beta} - 1\right) + \frac{(1 - \beta)t}{1 - \alpha t} - \frac{\beta (1 - \beta)t}{\beta + \alpha(1 - \beta)t}.
    \end{align}
\end{proposition}
\begin{proof}
    Noting that $q_{k, k, a}(x) = \One\{a = k\}$, we may slightly extend the limits of the summations inside $s_k(x)$ and write
    \begin{align*}
      S(x, t)
      &= \sum_{k \geq 1} \sum_{j = 0}^k \sum_{a = 0}^{j} t^k\left(\frac{x}{\beta} - 1\right)^{a} (-\alpha)^{k - j} q_{k, j, a}(\beta^{-1} - 1) - \sum_{k \geq 1}\alpha^{k - 1}\left(\sum_{b = 1}^{k - 1}(1 - \beta^{-1})^b\right) t^k \\
      &= Q\left(\frac{1}{\beta} - 1, -\alpha t, -\frac{1}{\alpha}, \frac{x}{\beta} - 1\right) - \sum_{k \geq 1}\alpha^{k - 1} \beta((1 - \beta^{-1}) - (1 - \beta^{-1})^k) t^k \\
      &= Q\left(\frac{1}{\beta} - 1, -\alpha t, -\frac{1}{\alpha}, \frac{x}{\beta} - 1\right) + \frac{(1 - \beta)t}{1 - \alpha t} - \frac{\beta (1 - \beta)t}{\beta + \alpha(1 - \beta)t}, \numberthis
    \end{align*}
    giving the result.
\end{proof}

The following is obtained by substituting the reparametrization described in Appendix~\ref{app:manova-param} into one of the results of \cite{DE-2015-InfiniteRandomMatrixTheory}.
\begin{proposition}[Table 7 of \cite{DE-2015-InfiniteRandomMatrixTheory}]
    \label{prop:de-orth-poly}
    Suppose $\alpha \leq \min\{\beta, 1 - \beta\}$.
    The sequence of polynomials $f_n(x)$ defined by $f_0(x) = 1$ and, for $n \geq 1$,
    \begin{align*}
      f_n(x) &= (x - \beta) \left(\sqrt{\alpha(1 - \beta)\beta(1 - \alpha)}\right)^{n - 1}U_{n - 1}\left(\frac{x - \alpha(1 - \beta) - \beta(1 - \alpha)}{2\sqrt{\alpha(1 - \beta)\beta(1 - \alpha)}}\right) \\
             &\hspace{0.7cm}- \frac{1}{1 - \alpha}\left(\sqrt{\alpha(1 - \beta)\beta(1 - \alpha)}\right)^{n} U_{n - 2}\left(\frac{x - \alpha(1 - \beta) - \beta(1 - \alpha)}{2\sqrt{\alpha(1 - \beta)\beta(1 - \alpha)}}\right) \numberthis
    \end{align*}
    form a sequence of monic orthogonal polynomials for $\DE(\alpha, \beta)$.
    Here $U_n(x)$ are the Chebyshev polynomials of the second kind, given recursively by $U_{-1}(x) = 0$, $U_0(x) = 1$, and $U_n(x) = 2xU_{n - 1}(x) - U_{n - 2}(x)$ for $n \geq 1$.
\end{proposition}

\begin{proposition}[Orthogonal polynomial generating function]
    \label{prop:de-gf}
    The $f_n(x)$ from Proposition~\ref{prop:de-orth-poly} admit the ordinary generating function
    \begin{equation}
        F(x, t) \colonequals \sum_{n = 0}^{\infty} f_n(x) t^n = \frac{1 + \alpha(1 - 2\beta)t - \alpha^2\beta(1 - \beta)t^2}{1 - t(x - \alpha(1 - \beta) - \beta(1 - \alpha)) + \alpha(1 - \beta)\beta(1 - \alpha)t^2}.
    \end{equation}
\end{proposition}
\begin{proof}
    We recall the standard fact (see, e.g., 22.9.10 of \cite{AS-1964-HandbookMathematicalFunctions}) that the Chebyshev polynomials of the second kind admit the ordinary generating function
    \begin{equation}
        \label{eq:cheb-gf}
        \sum_{n = 0}^{\infty} U_n(x) t^n = \frac{1}{1 - 2tx + t^2}.
    \end{equation}
    For the sake of brevity, let us write $\sigma \colonequals \sqrt{\alpha(1 - \alpha)\beta(1 - \beta)}$.
    We have
    \begin{align*}
      F(x, t)
      &= 1 + (x - \beta)t\sum_{n = 1}^{\infty}\sigma^{n - 1}U_{n - 1}\left(\frac{x - \alpha(1 - \beta) - \beta(1 - \alpha)}{2\sigma}\right)t^{n - 1} \\
      &\hspace{0.75cm} - \frac{\sigma^2}{1 - \alpha}t^2\sum_{n = 1}^{\infty} \sigma^{n - 2} U_{n - 2}\left(\frac{x - \alpha(1 - \beta) - \beta(1 - \alpha)}{2\sigma}\right)t^{n - 2} \\
      &= 1 + \left((x - \beta)t - \alpha\beta(1 - \beta)t^2\right)\sum_{n = 0}^{\infty} U_{n}\left(\frac{x - \alpha(1 - \beta) - \beta(1 - \alpha)}{2\sigma}\right) \left(\sigma t\right)^n \\
      &= 1 + \frac{(x - \beta)t - \alpha\beta(1 - \beta)t^2}{1 - t(x - \alpha(1 - \beta) - \beta(1 - \alpha)) + \sigma^2 t^2} \numberthis
    \end{align*}
    the last step following by \eqref{eq:cheb-gf}, and the result follows upon substituting back for $\sigma$ and rewriting.
\end{proof}

\begin{proof}[Proof of Lemma~\ref{lem:manova-recursion} when $\alpha \leq \min\{\beta, 1 - \beta\}$]
    Following our reasoning earlier, it suffices to show that
    \begin{equation}
        \Ex_{X \sim \DE(\alpha, \beta)} s_k(X) = 0
    \end{equation}
    for each $k \geq 1$.
    The key observation is the following relation of generating functions:
    \begin{equation}
        \label{eq:orth-poly-exp}
        S(x, t) = \left(F\left(x, \frac{t}{\beta}\right) - 1\right)\frac{\beta -\alpha(1 - \alpha)(1 - \beta)t^2}{\beta + \alpha(1 - 2\beta)t - \alpha^2(1 - \beta)t^2},
    \end{equation}
    which may be verified symbolically.
    Since $\EE_{X \sim \DE(\alpha, \beta)} F(X, t) = 1$ for all $t$ by the orthogonality of the $f_n(x)$, we then have $\EE_{X \sim \DE(\alpha, \beta)} S(X, t) = 0$, and expanding this as a power series in $t$ gives the claim.
\end{proof}

\begin{remark}
    This proof of course appears rather opaque.
    In discovering it, we first empirically discovered the orthogonal polynomial expansion that is encoded in \eqref{eq:orth-poly-exp}, then found the identity of generating functions that it would imply, and finally verified that identity algebraically.
\end{remark}

\section{Asymptotic liberation in Kesten-McKay case}
\label{app:asymp-lib}

In this appendix, we sketch how the results of \cite{AF-2014-AsymptoticallyLiberatingUnitaryMatrices} relate to the special case $\beta = \frac{1}{2}$ of our Theorem~\ref{thm:untf} (and thus also a variation of the main result of \cite{MMP-2019-RandomSubensembles}).
The general idea of \cite{AF-2014-AsymptoticallyLiberatingUnitaryMatrices} is to consider, in the setting of Theorem~\ref{thm:invariant-freeness}, what ``less random'' matrices than Haar-distributed unitaries can achieve ``asymptotic liberation,'' creating asymptotic freeness when they conjugate another sequence of matrices.
One version of their main result is the following.

\begin{theorem}[Corollary 3.7 of \cite{AF-2014-AsymptoticallyLiberatingUnitaryMatrices}]
    \label{thm:asymp-lib}
    Let $N = N(\iota)$ be an increasing sequence.
    Let $\bQ^{(\iota)} \in \sU(N)$ have $\sqrt{N} \cdot \max_{ij} |Q_{ij}^{(\iota)}|$ bounded as $\iota \to \infty$.
    Let $\bP^{(\iota)} \in \sU(N)$ be a uniformly random permutation matrix and $\bD^{(\iota)} \in \sU(N)$ be a diagonal matrix with i.i.d.\ diagonal entries $D_{ii}^{(\iota)} \sim \Unif(\{\pm 1\})$.
    Let $\bA^{(\iota)}, \bC^{(\iota)} \in \CC^{N \times N}$ be sequences of random diagonal matrices so that the sequences $(A^{(\iota)}_{ii}, B^{(\iota)}_{ii})$ are i.i.d.\ over all $\iota, i$ with bounded real-valued marginal laws.
    Then, the sequences
    \begin{equation}
        \bA^{(\iota)} \,\,\, \text{ and } \,\,\, \big(\bD^{(\iota)}\bP^{(\iota)}\bQ^{(\iota)}\bP^{(\iota)^*}\bD^{(\iota)}\big) \bC^{(\iota)} \big(\bD^{(\iota)}\bP^{(\iota)}\bQ^{(\iota)^*}\bP^{(\iota)^*}\bD^{(\iota)}\big)
    \end{equation}
    are asymptotically free.
\end{theorem}

The relevance of this to our setting is as follows.
Consider $\bA = \bC$ with i.i.d.\ diagonal entries distributed as $\Ber(\alpha)$ (let us drop the $\iota$ superscript for the sake of brevity).
We then may compute as a free multiplicative convolution the law of
\begin{align*}
  \bA^{1/2}\big(\bD\bP\bQ\bP^*\bD\big) \bA \big(\bD\bP\bQ^*\bP^*\bD\big)\bA^{1/2}
  &= \bA\big(\bD\bP\bQ\bP^*\bD\big) \bA \big(\bD\bP\bQ^*\bP^*\bD\big)\bA \\
  &= \bA\bD\bP\bQ(\bP^*\bA\bP) \bQ^*\bP^*\bD\bA \\
  &= (\bA\bD\bP\bQ(\bP^*\bA\bP)) ((\bP^*\bA\bP)\bD\bP\bQ\bA)^*, \numberthis
\end{align*}
which has the same non-zero spectrum as
\begin{align*}
  &\hspace{-1cm} (\bA\bD\bP\bQ(\bP^*\bA\bP))^*(\bA\bD\bP\bQ(\bP^*\bA\bP)) \\
  &= (\bP^*\bA\bP)\bQ^*\bP^*\bD\bA^2\bD\bP\bQ(\bP^*\bA\bP) \\
  &= (\bP^*\bA\bP)\bQ^*(\bP^*\bA\bP)\bQ(\bP^*\bA\bP) \\
  &= ((\bP^*\bA\bP)\bQ(\bP^*\bA\bP))^*((\bP^*\bA\bP)\bQ(\bP^*\bA\bP)), \numberthis
\end{align*}
which, since the diagonal entries of $\bA$ are i.i.d., has the same law as simply
\begin{equation}
    \bA\bQ^*\bA\bQ\bA = (\bA\bQ\bA)^*(\bA\bQ\bA).
\end{equation}
Thus, this result of \cite{AF-2014-AsymptoticallyLiberatingUnitaryMatrices} allows us to compute the limiting \emph{singular value} distribution of a random submatrix of any sequence of sufficiently incoherent (in the sense of having small off-diagonal entries) unitary matrices $\bQ$.
This law will be the free multiplicative convolution $\Ber(\alpha) \boxtimes \Ber(\alpha) = \MANOVA(\alpha, \alpha)$.

To relate this to our setting, note that, for $\bB$ an orthogonal projection, $\bQ \colonequals 2\bB - \bm I_N$ is a unitary matrix---it is the reflection across the hyperplane to which $\bB$ projects.
When the assumption of Theorem~\ref{thm:untf} holds with $\beta = \frac{1}{2}$, then Theorem~\ref{thm:asymp-lib} applies to this $\bQ$.
The above then lets us compute the limiting law of the squared eigenvalues of $\bA(2\bB - \bm I_N)\bA = 2\bA\bB\bA - \bA$ as $\MANOVA(\alpha, \alpha)$.
Then, a version of Theorem~\ref{thm:untf} for convergence in moments of the limiting empirical distribution of a ``folding'' of the eigenvalues of $\bA\bB\bA$ under the map $x \mapsto (2x - 1)^2$ follows from the following curious distributional identity, along with accounting for the masses of the atoms on 0 and 1 of the two MANOVA laws.

\begin{proposition}
    Let $\alpha \in (0, 1)$ and $X \sim \MANOVA(\frac{1}{2}, \alpha)$ conditioned to equal neither 0 nor 1.
    Then, $(2X - 1)^2$ has the law $\MANOVA(\alpha, \alpha)$ conditioned to equal neither 0 nor 1.
\end{proposition}

We find this identity quite surprising, since, appealing to the random matrix interpretation of the MANOVA laws, \emph{a priori} there seems to be no reason for the laws of singular values of a $\frac{1}{2}N \times \alpha N$ and a square $\alpha N \times \alpha N$ random submatrix of a Haar-distributed unitary matrix to be related.
It would be interesting to understand if there is a more intuitive free probability explanation of this, and whether there are larger families of distributional identities among MANOVA laws.

\section{Combinatorial tools}

\subsection{Stirling number bounds}
In this appendix, we prove some combinatorial bounds used in the proof of Theorem~\ref{thm:edge-km-unitary} in the main text.
These involve sums over partitions or permutations, which are conveniently expressed in terms of Stirling numbers.

\begin{definition}
    The \emph{ordinary} and \emph{associated Stirling numbers} of the \emph{first} and \emph{second kind} are defined as follows:
    \begin{itemize}
    \item The \emph{Stirling numbers of the first kind}, denoted $s(n, k)$, give the number of permutations in $S_n$ having exactly $k$ cycles.
    \item The \emph{associated Stirling numbers of the first kind}, denoted $s_r(n, k)$, give the number of permutations in $S_n$ having exactly $k$ cycles, each of which has length at least $r + 1$.
    \item The \emph{Stirling numbers of the second kind}, denoted $S(n, k)$, give the number of partitions in $\Part([n])$ having exactly $k$ parts.
    \item The \emph{associated Stirling numbers of the second kind}, denoted $S_r(n, k)$, give the numbers of partitions in $\Part([n])$ having exactly $k$ parts, each of which has size at least $r + 1$.
    \end{itemize}
\end{definition}

The following results are well-known descriptions of various special cases of the Stirling numbers; they may all be found in, for example, \cite{Howard-1980-AssociatedStirlingNumbers}.
\begin{proposition}
    \label{prop:S-falling}
    For all $n \geq 1$, $x^n = \sum_{k = 1}^n S(n, k) x^{\underline{k}}$.
\end{proposition}

\begin{proposition}
    \label{prop:S1-rec}
    For all $n \geq 2$ and $k \geq 1$, $S_1(n, k) = kS_1(n - 1, k) + nS_1(n - 2, k - 1)$.
\end{proposition}

\begin{proposition}
    \label{prop:s1-rec}
    For all $n \geq 2$ and $k \geq 1$, $s_1(n, k) = ns_1(n - 1, k) + ns_1(n - 2, k - 1)$.
\end{proposition}

\begin{proposition}
    \label{prop:S-sum}
    Suppose that $x \geq 2\ell$.
    Then,
    \begin{equation}
        \sum_{\pi \in \Part([\ell])} x^{|\pi|} \leq e^{\ell^2 / x} x^{\ell}.
    \end{equation}
\end{proposition}
\begin{proof}
    First, suppose $1 \leq a \leq \ell$.
    Then, $\frac{a}{x} \leq \frac{\ell}{x} \leq \frac{1}{2}$.
    Noting that for all $0 \leq t \leq \frac{1}{2}$ we have $1 - t \geq e^{-2t}$, we may bound
    \begin{equation}
        \frac{x^a}{x^{\underline{a}}} = \frac{1}{(1 - \frac{0}{x})(1 - \frac{1}{x}) \cdots (1 - \frac{a - 1}{x})} \leq \exp\left(2\sum_{b = 1}^{a - 1} \frac{b}{x}\right) \leq \exp\left(\frac{a^2}{x}\right) \leq \exp\left(\frac{\ell^2}{x}\right).
    \end{equation}
    Rewriting in terms of Stirling numbers and using Proposition~\ref{prop:S-falling}, we have
    \begin{align*}
      \sum_{\pi \in \Part([\ell])} x^{|\pi|}
      &= \sum_{a = 1}^{\ell}S(\ell, a) x^a \\
      &\leq e^{\ell^2 / x} \sum_{a = 1}^{\ell}S(\ell, a) x^{\underline{a}} \\
      &= e^{\ell^2 / x} x^{\ell}, \numberthis
    \end{align*}
    completing the proof.
\end{proof}

\begin{proposition}
    \label{prop:s1-sum}
    Suppose that $x \geq 16\ell^{3/2}$.
    Then,
    \begin{equation}
        \sum_{\substack{\sigma \in S_{\ell} \\ \cyc_1(\sigma) = \emptyset}} x^{|\cyc(\sigma)|} \leq \exp(\ell^{3/4})\,  \ell !! \, x^{\ell / 2}.
    \end{equation}
\end{proposition}
\begin{proof}
    Let us view $x$ as fixed, and define the function of $\ell$
    \begin{equation}
        f(\ell) \colonequals \sum_{\substack{\sigma \in S_{\ell} \\ \cyc_1(\sigma) = \emptyset}} x^{|\cyc(\sigma)|} = \sum_{k = 1}^{\ell} s_1(\ell, k)x^k.
    \end{equation}
    Applying the recursion from Proposition~\ref{prop:s1-rec}, we have
    \begin{align*}
      f(\ell)
      &= \ell \sum_{k = 1}^{\ell - 1} s_1(\ell - 1, k) x^k + \ell \sum_{k = 1}^{\ell - 2} s_1(\ell - 2, k - 1) x^{k} \\
      &= \ell f(\ell - 1) + \ell x f(\ell - 2). \numberthis
    \end{align*}
    We now verify that $f$ satisfies the bound claimed by induction.
    We have $f(1) = 0$ and $f(2) = x$, so the bound clearly holds for $\ell = 1, 2$.
    Supposing that the bound holds for all $f(\ell^{\prime})$ with $\ell^{\prime} < \ell$, we have, continuing from above, for $\ell \geq 3$
    \begin{align*}
      f(\ell)
      &\leq \ell \exp((\ell - 1)^{3/4}) (\ell - 1)!! \, x^{(\ell - 1) / 2} + \ell x \exp((\ell - 2)^{3/4}) (\ell - 2)!! x^{(\ell - 2) / 2} \\
      &\leq \exp(\ell^{3/4})\, \ell!! \, x^{\ell / 2} \left( \frac{\ell}{\sqrt{x}} \frac{(\ell - 1)!!}{\ell!!}\exp((\ell - 1)^{3/4} - \ell^{3/4}) + \exp((\ell - 2)^{3/4} - \ell^{3/4}) \right)
        \intertext{and, using that $(\ell - 1)!! / \ell!! \leq \frac{4}{\sqrt{\ell}}$ and the assumption $x \geq 16\ell^{3/2}$,}
      &\leq \exp(\ell^{3/4}) \, \ell!! \, x^{\ell / 2} \left( \ell^{-1/4} \exp((\ell - 1)^{3/4} - \ell^{3/4}) + \exp((\ell - 2)^{3/4} - \ell^{3/4})\right).
    \end{align*}
    The result then follows from the elementary inequality that $\exp(\ell^{3/4}) - \ell^{-1/4}\exp((\ell - 1)^{3/4}) - \exp((\ell - 2)^{3/4}) \geq 0$ for all $\ell \geq 2$.
    We point out that we have arranged the powers appearing in the condition on $x$ and in the $\exp(\ell^{3/4})$ prefactor in the result based on the fact that $\frac{d}{dt}\exp(t^{3/4}) = \frac{3}{4}t^{-1/4} \exp(t^{3/4})$, which makes the above inequality plausible.
\end{proof}

\begin{proposition}
    \label{prop:S1-sum}
    Suppose that $x \geq 16\ell^{3/2}$.
    Then,
    \begin{equation}
        \sum_{\substack{\pi \in \Part([\ell]) \\ |S| \geq 2 \text{ for all } S \in \pi}} x^{|\pi|} \leq \exp(\ell^{3/4}) \, \ell !! \, x^{\ell / 2}.
    \end{equation}
\end{proposition}
\begin{proof}
    The left-hand side of this Proposition is bounded above by the left-hand side of Proposition~\ref{prop:s1-sum}, so this result follows immediately.
\end{proof}

\begin{comment}
\begin{proposition}
    $S(n, k) \leq e^k n! / k!$.
\end{proposition}

\begin{proposition}
    Suppose that $x \geq 2/e \cdot \ell$.
    Then,
    \begin{equation}
        \sum_{\pi \in \Part([\ell])} x^{|\pi|} \leq 2(ex)^{\ell}.
    \end{equation}
\end{proposition}
\begin{proof}
    Rewriting in terms of Stirling numbers,
    \begin{align}
      \sum_{\pi \in \Part([\ell])} x^{|\pi|}
      &= \sum_{a = 1}^{\ell}S(\ell, a) x^a \\
      &\leq \ell! \sum_{a = 1}^{\ell}\frac{1}{a!} (ex)^a \\
      &\leq \ell! \cdot 2\frac{(ex)^{\ell}}{\ell!} \\
      &= 2(ex)^{\ell}.
    \end{align}
\end{proof}
\end{comment}

\subsection{Riordan arrays: Proof of Proposition~\ref{prop:riordan-inv}}
\label{app:riordan}

We give a proof of Proposition~\ref{prop:riordan-inv} using the theory of Riordan arrays, certain infinite triangular matrices satisfying a generating function identity that allows for matrix multiplications to be carried out in terms of generating functions alone.
We introduce only the basic notions of these arrays here; the interested reader may consult \cite{SSBCHMW-1991-RiordanGroup} for more information about this beautiful theory.

\begin{definition}
    A \emph{Riordan array} is an infinite lower-triangular matrix $\bA$ such that there exist formal power series $F(z)$ and $G(z)$ with $G(z)F(z)^k$ giving the ordinary generating function of the entries of the $k$th column of $\bA$ (with columns indexed by $k \geq 0$).
    In this case, we call $(G(z), F(z))$ the \emph{parameters} of the Riordan array.
\end{definition}
\noindent
The following is the main property of Riordan arrays alluded to above.

\begin{theorem}
    \label{thm:riordan-comp}
    Suppose $\bA$ is a Riordan array with parameters $(g(z), f(z))$ and $\bB$ is a Riordan array with parameters $(G(z), F(z))$.
    Then, $\bC = \bA\bB$ (viewed as a multiplication of infinite matrices, which however only involves finite summations) is a Riordan array with parameters
    \begin{equation}
        (g(z) \cdot G(f(z)), F(f(z))).
    \end{equation}
\end{theorem}
\noindent
In fact, the Riordan arrays form a group of infinite matrices, whose multiplication may be conducted purely in terms of generating functions.
In particular, the identity matrix is the Riordan array with parameters $(1, z)$.

The following four Propositions are the main ingredients in our proof of Proposition~\ref{prop:riordan-inv}, identifying various submatrices of the matrices involved as Riordan arrays with particular parameters.
We leave the verification of these facts as an exercise.
\begin{proposition}
    \label{prop:riordan-1-even}
    The lower-triangular array with entries $(T(\ell + b, \ell - b))_{0 \leq b \leq \ell}$ is the Riordan array with parameters
    \begin{equation}
        \left( \frac{1 + z}{1 - z}, \frac{z}{(1 - z)^2} \right).
    \end{equation}
\end{proposition}

\begin{proposition}
    \label{prop:riordan-1-odd}
    The lower-triangular array with entries $(T(\ell + b + 1, \ell - b))_{0 \leq b \leq \ell}$ is the Riordan array with parameters
    \begin{equation}
        \left( \frac{1 + z}{(1 - z)^2}, \frac{z}{(1 - z)^2} \right).
    \end{equation}
\end{proposition}

\begin{proposition}
    \label{prop:riordan-2-even}
    The lower-triangular array with entries $((-1)^{\ell + b} \binom{2\ell}{\ell + b})_{0 \leq b \leq \ell}$ is the Riordan array with parameters
    \begin{equation}
        \left( \frac{1}{\sqrt{1 + 4z}}, 1 + \frac{1 - \sqrt{1 + 4z}}{2z}\right).
    \end{equation}
\end{proposition}

\begin{proposition}
    \label{prop:riordan-2-odd}
    The lower-triangular array with entries $((-1)^{\ell + b + 1} \binom{2\ell + 1}{\ell + b + 1})_{0 \leq b \leq \ell}$ is the Riordan array with parameters
    \begin{equation}
        \left( \frac{-1 + \sqrt{1 + 4z}}{2z\sqrt{1 + 4z}}, 1 + \frac{1 - \sqrt{1 + 4z}}{2z}\right).
    \end{equation}
\end{proposition}

\begin{proof}[Proof of Proposition~\ref{prop:riordan-inv}]
    The infinite matrices described in the Proposition are, partitioning the rows and columns by parity of the index, the direct sums of the matrices described in Propositions~\ref{prop:riordan-1-even} and \ref{prop:riordan-1-odd} and in Propositions~\ref{prop:riordan-2-even} and \ref{prop:riordan-2-odd}, respectively.
    Thus it suffices to check that the matrix in Proposition~\ref{prop:riordan-1-even} is the inverse of that in Proposition~\ref{prop:riordan-2-even}, and that the matrix in Proposition~\ref{prop:riordan-1-odd} is the inverse of that in Proposition~\ref{prop:riordan-2-odd}.
    This is a straightforward generating function calculation using Theorem~\ref{thm:riordan-comp}, along with the observation that the identity matrix is the Riordan array with parameters $(1, z)$.
\end{proof}

\end{document}